%% file: main.tex
\documentclass{article}

\usepackage[ansinew, utf8]{inputenc}
\usepackage{amsmath, amssymb, amsthm}
\usepackage{multicol,latexsym, array}
\usepackage{subfigure}
\usepackage{paralist}
\usepackage{ dsfont }
\usepackage{color}
\usepackage{mathrsfs}
\usepackage{natbib}
\usepackage{bm}
\usepackage[colorlinks,citecolor=blue,urlcolor=blue]{hyperref}
\usepackage{graphicx}
\usepackage{tikz}
\usetikzlibrary{fit,shapes,decorations}

\usepackage{wrapfig}
\providecommand{\keywords}[1]{\textit{Keywords} #1}
\providecommand{\MSC}[1]{MSC \textit{subject classification} #1}

\newcommand{\R}{\mathbb{R}}
\newcommand{\N}{\mathbb{N}}
\newcommand{\Z}{\mathbb{Z}}

\newcommand{\one}{\mathds{1}}
\newcommand{\mean}{\mathds{E}}

\newcommand{\calD}{\mathcal{D}}

\newcommand{\calN}{\mathcal{N}}

\newcommand{\calP}{\mathcal{P}}
\newcommand{\calS}{\mathcal{S}}
\newcommand{\calT}{\mathcal{T}}
\newcommand{\calU}{\mathcal{U}}
\newcommand{\calX}{\mathcal{X}}
\newcommand{\calY}{\mathcal{Y}}

\newcommand{\bG}{\mathbf{G}}
\newcommand{\bH}{\mathbf{H}}

\newcommand{\bd}{\bm{d}}
\newcommand{\be}{\bm{e}}

\newcommand{\br}{\bm{r}}
\newcommand{\bs}{\bm{s}}
\newcommand{\bu}{\bm{u}}
\newcommand{\bv}{\bm{v}}
\newcommand{\bw}{\bm{w}}

\newcommand{\bz}{\bm{z}}
\newcommand{\bmu}{\bm{\mu}}
\newcommand{\bnu}{\bm{\nu}}
\newcommand{\blambda}{\bm{\lambda}}
\newcommand{\boeta}{\bm{\eta}}

\newcommand{\bnull}{\bm{0}}

\newcommand{\brh}{\hat{\mathbf{r}}}
\newcommand{\bsh}{\hat{\mathbf{s}}}

\newcommand{\set}[1]{\left\lbrace #1\right\rbrace }
\newcommand{\inter}[1]{int\left\lbrace #1\right\rbrace }
\newcommand{\defeq}{\mathrel{\mathop{:}}=}
\newcommand{\scalprod}[1]{\langle #1\rangle }
\newcommand{\abs}[1]{\left|  #1\right|  }
\newcommand{\norm}[1]{\|  #1\|  }
\newcommand{\normfp}[1]{\|  #1\|_{\ell^1(d^p)}  }
\newcommand{\normwinfty}[1]{\|  #1\|_{\ell^\infty(1/d^p)}  }
\newcommand{\weak}{\xrightarrow{\mathscr{D}}}

\newcommand{\Wt}{W^t}
\newcommand{\parent}{{\mathrm{parent}}}
\newcommand{\children}{{\mathrm{children}}}
\newcommand{\rootT}{{\mathrm{root}}}

\newcommand{\supp}{\mathrm{supp}}

\newtheorem{theorem}{Theorem}[section]
\newtheorem{definition}[theorem]{Definition}
\newtheorem{example}[theorem]{Example}

\newtheorem{lem}[theorem]{Lemma}
\newtheorem{rem}[theorem]{Remark}

\newcommand{\footremember}[2]{%
	\footnote{#2}
	\newcounter{#1}
	\setcounter{#1}{\value{footnote}}%
}
\newcommand{\footrecall}[1]{%
	\footnotemark[\value{#1}]%
}

\begin{document}

\author{Carla Tameling \footremember{ims}{\scriptsize Institute for Mathematical
		Stochastics, University of G\"ottingen,
		Goldschmidtstra{\ss}e 7, 37077 G\"ottingen} 
	\and 
	Max Sommerfeld \footrecall{ims}\footremember{fbms}{\scriptsize Felix Bernstein Institute for
		Mathematical Statistics in the Biosciences, University of G\"ottingen,
		Goldschmidtstra{\ss}e 7, 37077 G\"ottingen} 
	\and 
	Axel Munk \footrecall{ims}\footrecall{fbms} \footnote{\scriptsize Max Planck Institute for Biophysical
		Chemistry, Am Fa{\ss}berg 11, 37077 G\"ottingen}}

	\title{Empirical optimal transport on countable metric spaces: Distributional limits and statistical applications}

	\maketitle
	\begin{abstract}
	\input{section/abstract}	
	\end{abstract}
	
		\MSC{Primary 60F05, 60B12, 62E20; Secondary 90C08, 90C31, 62G10}
	
	\keywords{optimal transport, Wasserstein distance, empirical process, limit law, statistical testing}

\input{section/Introduction}

\input{section/CLT}

\input{section/tree}

\input{section/rates}

\input{section/SMS}

\section*{Acknowledgments}
The authors gratefully acknowledge support by the DFG Research Training Group 2088 Project A1 and CRC 755 Project A6. They would like to thank M. Klatt for careful reading of the manuscript.  A. Munk is grateful to helpful comments of J. Wellner.

\bibliographystyle{apalike}
\bibliography{discretewasser,finitewasser}

\input{section/appendix}

\end{document}

%% file: section/abstract.tex
We derive distributional limits for empirical transport distances between
probability measures supported on countable sets. Our approach is based on
sensitivity analysis of optimal values of infinite dimensional mathematical
programs and a delta method for non-linear derivatives. A careful calibration of
the norm on the space of probability measures is needed in order to combine
differentiability and weak convergence of the underlying empirical process.
Based on this we provide a sufficient and necessary condition for the underlying
distribution on the countable metric space for such a distributional limit to
hold.
We give an explicit form of the limiting distribution for ultra-metric spaces.\\
Finally, we apply our findings to optimal transport based inference  in
large scale problems. An application to nanoscale microscopy is given.\\ 

%% file: section/Introduction.tex
\section{Introduction}
\label{sec:intro}

Optimal transport based distances between probability measures (see e.g., \cite{rachev_mass_1998} or
\cite{villani_optimal_2008} for a comprehensive treatment), e.g., the Wasserstein
distance \citep{vasershtein_markov_1969}, which is also known as Earth
Movers distance \citep{rubner_earth_2000}, Kantorovich-Rubinstein distance
\citep{kantorovich_space_1958} or Mallows distance \citep{mallows_note_1972},
are of fundamental interest in probability and statistics, with respect to both
theory and practice. 
The $p$-th Wasserstein distance (WD) between two probability measures $\mu$ and
$\nu$ on a Polish metric space $(\calX,d)$ is given by 
\begin{equation}
W_p(\mu,\nu) = \left( \inf_{\pi} \int_{\calX \times \calX} d(x,y)^p d\pi(x,y)\right) ^{1/p}
\label{eq:def_wasser}
\end{equation}
for $p\in[1,\infty)$, the infimum is taken over all probability measures $\pi$
  on the product space $\calX \times \calX$ with marginals $\mu$ and $\nu$. \\
The WD metrizes weak convergence of a sequence of probability measures on $(\calX,d)$ together with convergence of its first $p$ moments and has become a standard tool in probability, e.g., to
study limit laws (e.g.,
\cite{johnson_central_2005,rachev_rate_1994,shorack_empirical_1986}), to derive
bounds for Monte Carlo computation schemes such as MCMC (e.g.,
\cite{eberle_error_2014, rudolf_perturbation_2015}), for point process
approximations \citep{barbour_stein_1992,schuhmacher_stein_2009}, bootstrap
convergence \citep{bickel_asymptotic_1981} or to quantify measures of risk \citep{rachev_probability_2011}. 
Besides of its theoretical importance, the WD is used in many applications as an empirical measure to compare complex objects, e.g., in image retrieval
\citep{rubner_earth_2000}, deformation analysis \citep{panaretos_amplitude_2016},
meta genomics \citep{evans_phylogenetic_2012}, computer vision
\citep{ni_local_2009}, goodness-of-fit testing
\citep{Munk1998,barrio_contributions_2000} and machine learning \citep{Rolet2016}.\\

In such applications the WD has to be estimated from a finite sample of
the underlying measures.
This raises the question how fast the \emph{empirical} Wasserstein distance (EWD), i.e., when either $\mu$ or $\nu$ (or both) are estimated by the empirical measures $\boldmath{\hat{\mu}_n} = \frac{1}{n}\sum_{i=1}^{n}\delta_{X_i}$ (and $\boldmath{\hat{\nu}_m} = \frac{1}{m}\sum_{i=1}^{m}\delta_{Y_i}$) approaches WD. 
\cite{ajtai_optimal_1984} investigated the rate of convergence of EWD for the uniform measure on the unit square, \cite{talagrand_matching_1992} and \cite{talagrand_transportation_1994} extended this to higher dimensions. \cite{Horowitz1994} then provided non-asymptotic bounds for the average speed of convergence for the empirical 2-Wasserstein distance. There are several refinements of these results, e.g., \cite{Boissard2014}, \cite{fournier_rate_2014} and \cite{weed_sharp_2017}.

As a natural extension of such results, there is a long
standing interest in distributional limits for EWD, in particular motivated from
statistical applications. Most of this work is restricted to the univariate
case $\calX \subset \R$. 
 \cite{Munk1998} derived central limit theorems for a trimmed WD on the real
 line when $\mu \neq \nu$ whereas \cite{Barrio1999, delBarrio_clt_1999} consider
 the empirical Wasserstein distance when $\mu$ belongs to a parametric family of
 distributions for the assessment of
 goodness of fit,  e.g., for a Gaussian location scale family. In a similar spirit \cite{Barrio2005} provided asymptotics for
 a weighted version of the empirical 2-Wasserstein distance in one dimension and
 \cite{Freitag2005} derive limit laws for semiparametric models, still
 restricted to the univariate case. There are also several results for dependent
 data in one dimension, e.g., \cite{dede_empirical_2008},
 \cite{dedecker_behavior_2015}. For a recent survey we refer to
 \cite{Bobkov2014} and \cite{mason_weighted_2016} and references therein.
 A major reason of the limitation to dimension $D = 1$ is that only for $\calX
 \subset \R$ (or more generally a totally ordered space) the coupling which
 solves \eqref{eq:def_wasser} is known explicitly
 and can be expressed in terms of the quantile functions $F^{-1}$ and $G^{-1}$
 of $\mu$ and $\nu$, respectively, as $\pi = (F^{-1} \times G^{-1}) \#
 \mathcal{L}$, where $\mathcal{L}$ is the Lebesgue measure on $[0,1]$ (see
 \cite{mallows_note_1972}). All the above mentioned work relies essentially on
 this fact. For higher dimensions only in specific settings such a coupling can
 be computed explicitly and then can be used to derive limit laws
 \citep{rippl_limit_2016}. Already for $D = 2$ \cite{ajtai_optimal_1984}
 indicate that the scaling rate for the limiting distribution of
 $W_1(\hat{\mu}_n, \mu)$ when $\mu$ is the uniform measure on $\calX = [0,1]^2$
 (if it exists) must be of complicated nature as it is bounded from above and
 below by a rate of order $\sqrt{n \log(n)}$. 

Recently, \cite{del_barrio_central_2017} gave  distributional limits for the
quadratic EWD in general dimension with a scaling rate $\sqrt{n}$. This yields a (non-degenerate) normal limit in the case $\mu \neq \nu$, i.e., when the
data generating measure is different from the measure to be compared with (extending
\cite{Munk1998} to $D > 1$). Their result centers the EWD with an expected EWD
(whose value is typically unknown) instead of the true WD and  requires $\mu$ and $\nu$ to have a
positive Lebesgue density on the interior of their convex support.  Their proof
uses the uniqueness and stability of the
optimal transportation potential (i.e., the minimizer of the dual transportation
problem, see \cite{villani_topics_2003} for a definition and further results)
and the Efron-Stein variance inequality. However, in the case $\mu = \nu$, their
distributional limit degenerates to a point mass at $0$, underlining the fundamental difficulty of this problem again.  

An alternative approach has been advocated recently in \cite{sommerfeld_inference_2018}
who restrict to finite spaces $\calX = \left\lbrace x_1,\ldots,
x_N\right\rbrace$. They derive limit laws for the EWD for $\mu = \nu$ (and $\mu
\neq\nu$), which requires a different scaling rate. In this
paper we extend their work to measures $\br = (r_x)_{x \in
\calX}$ that are supported on countable metric spaces $(\calX, d)$. Our approach links the
asymptotic distribution of the EWD on the one hand to the issue of weak convergence of the
underlying multinomial process associated with $\hat{\mu}_n$ with respect to a weighted $\ell^1$-norm (for fixed, but arbitrary $x_0 \in \calX$)
 \begin{equation}
 \normfp{\br} = \sum_{x \in \calX} d^p(x,x_0) \abs{r_x} + \abs{r_{x_0}},
 \label{eq:norm}
 \end{equation} 
 and on the other hand to infinite dimensional sensitivity analysis of the underlying linear program. Notably, we obtain a necessary and sufficient
 condition for such a limit law, which sheds some light on the limitation to
 approximate the WD between continuous measures for $D \geq 2$ by discrete random variables.\\

The outline of this paper is a follows. In Section \ref{sec:CLT} we give distributional limits for the EWD of measures that are supported on a countable metric space. In short, this limit can be characterized as the optimal value of an infinite dimensional linear program applied to a Gaussian process over the set of dual solutions. The main ingredients of the proof are the directional Hadamard differentiability of the Wasserstein distance on countable metric spaces and the delta method for non-linear derivatives. We want to emphasize that the delta method for non-linear derivatives is not a standard tool (see \cite{shapiro_asymptotic_1991, romisch_delta_2004}). Moreover, for the delta method to work here weak convergence in the weighted $\ell^1$-norm \eqref{eq:norm} of the underlying empirical process $\sqrt{n}(\brh_n - \br)$ is required as the directional Hadamard differentiability is proven w.r.t. this norm. We cannot prove the directional Hadamard differentiability with our methods w.r.t. the $\ell^1$-norm as the space of probability measures with finite $p$-th moment is not complete with respect to the $\ell^1$-norm, see Section \ref{sub:bounded_dia} for more details.
 We find that \begin{equation}
\sum_{x \in \calX} d^p(x,x_0)\sqrt{r_x} < \infty
\label{eq:entropy}
\end{equation} is necessary and sufficient for weak convergence. This condition arises from Jain's CLT \citep{Jain1977}. Furthermore, we examine \eqref{eq:entropy} in a more detailed way in Section \ref{sub:summability}. We give examples and counterexamples for \eqref{eq:entropy} and discuss whether the condition holds in case of an approximation of continuous measures. Further, we examine under which assumptions it follows that \eqref{eq:entropy} holds for all $p' \leq p$ if it is fulfilled for $p$, and put it in relation to its one-dimensional counterpart, see \cite{delBarrio_clt_1999}. We close this section by discussing simplifications for ground spaces $\calX$ with bounded diameter. \\
In Section \ref{sec:tree} we specify the case where the metric structure on the ground space is given by a rooted tree with weighted edges. In this case we can provide a simplified limiting distribution and use its explicit formula to derive a distributional upper bound for general metric spaces.\\ 
 In Section \ref{sec:rates} we combine this with a well known lower bound \citep{pele_fast_2009} to derive a computationally efficient strategy to test for the equality of two measures $\br$ and $\bs$ on a countable metric space. Furthermore, we derive an explicit formula of the upper bound from Section \ref{sec:tree} in the case of the support of $\br$ being a regular grid.\\
An application of our results to data from single marker switching microscopy imaging is given in Section \ref{sec:appl_SMS}. As the number of pixels typically is of magnitude $10^5$ - $10^6$ this challenges the assumptions of a finite space underlying the limit law in \cite{sommerfeld_inference_2018} and our work provides the theoretical justification to perform EWD based inference in such a case. Finally, we stress that our results can be extended to many other situations, e.g., the comparison of $k$ samples and when the underlying data are dependent, as soon as a weak limit of the underlying empirical process w.r.t. the weighted $\ell^1$-norm \eqref{eq:norm} can be shown.

%% file: section/CLT.tex
\section{Distributional Limits}
\label{sec:CLT}
\subsection{Wasserstein distance on countable metric spaces}
Let throughout the following $\calX = \left\lbrace x_1, x_2, \ldots\right\rbrace $ be a countable metric space equipped with a metric $d \colon \calX \times \calX \to \R_+$. The probability measures on $\calX$ are infinite dimensional vectors $\br$ in 
\[\calP(\calX) = \left\lbrace \br = (r_x)_{x \in \calX}: r_x \geq 0 \quad \forall x \in \calX \text{ and } \sum_{x \in \calX} r_x = 1\right\rbrace. \] 
We want to emphasize that we consider the discrete topology on $\calX$ and do not embed $\calX$ for example in $\R^d$. This implies that the support of any probability measure $\br \in \calP(\calX)$ is the union of points $x \in \calX$ such that $r_x > 0$.
The $p$-th Wasserstein distance ($p \geq 1$) then becomes
\begin{equation}
W_p(\br,\bs) = \left\lbrace \min_{\bw \in \Pi(\br,\bs)}\sum_{x,x' \in \calX} d^p(x,x')w_{x,x'}\right\rbrace^{1/p},
\label{eq:wasser}
\end{equation}
where 
\begin{multline*}
\Pi(\br,\bs) = \Big\lbrace \bw \in \calP(\calX \times \calX): \sum_{x'\in\calX} w_{x,x'} = r_x \\
\text{ and } \sum_{x\in\calX} w_{x,x'} = s_{x'} \quad \forall x,x'\in \calX\Big\rbrace\end{multline*}
is the set of all couplings between $\br$ and $\bs$.
Furthermore, let
\[\calP_p(\calX) = \left\lbrace \br\in \calP(\calX): \sum_{x\in\calX} d^p(x,x_0)r_x < \infty \right\rbrace \]
be the set of probability measures on the countable metric space $\calX$ with finite $p$-th moment w.r.t. $d$. Here, $x_0 \in \calX$ is arbitrary and we want to mention that the space is independent of the choice of $x_0$. We need to introduce the weighted $\ell^1$-space $\ell^1_{d^p}(\calX)$ which is defined via the weighted $\ell^1$-norm \eqref{eq:norm} as in this case the set of probability measures with finite $p$-th moment is a closed subset and hence complete itself. This will play a crucial role in the proof of the directional Hadamard differentiability (see Appendix \ref{sec:hadamard}). The weighted $\ell^1$-norm \eqref{eq:norm} can be extended in the following way to sequences on $\calX \times \calX$ and hence to $\calP_p(\calX \times \calX)$
\begin{multline*}
\normfp{\bw} = \sum_{x, x'\in\calX} d^p(x_0,x) \abs{w_{x,x'}} + \abs{w_{x_0,x'}} \\
+ \sum_{x, x'\in\calX} d^p(x_0,x') \abs{w_{x,x'}} + \abs{w_{x,x_0}}.\end{multline*}

\subsection{Main Results}
Before we can state the main results we need a few definitions.\\
Define the empirical measure generated by i.i.d. random variables $X_1, \ldots, X_n$ from the measure $\br$ as
\begin{equation}
\brh_n = (\hat{r}_{n,x})_{x\in\calX}, \text{ where } \hat{r}_{n,x} = \frac{1}{n}\sum_{k=1}^{n}\one_{\left\lbrace X_k = x\right\rbrace },
\label{eq:emp_measure}
\end{equation}
and $\bsh_m$ is defined in the same way by $Y_1,\ldots, Y_m \overset{i.i.d.}{\sim} \bs$.
In the following we will denote weak convergence by $\weak$ and furthermore, let 
\[\ell^{\infty}(\calX) = \set{(a_x)_{x \in \calX} \in \R^\calX: \sup_{x \in \calX} \abs{a_x} < \infty}\]
and 
\[\ell^{1}(\calX) = \set{(a_x)_{x \in \calX} \in \R^\calX: \sum_{x \in \calX} \abs{a_x} < \infty}.\]
Finally, we also require a weighted version of the $\ell^\infty$-norm to characterize the set of dual solutions:
\[\normwinfty{a} = \max\left(\abs{a_{x_0}}, \sup_{x \neq x_0 \in\calX} \abs{d^{-p}(x,x_0) a_x}\right),\]
for $p \geq 1$.
The space $\ell^\infty_{d^{-p}}(\calX)$ contains all elements which have a finite $\normwinfty{\cdot}$-norm.\\
	For $\br,\bs \in \calP_p(\calX)$ we define the following convex sets
	\begin{equation}
	\begin{aligned}
	\calS^*(\br,\bs) = \Big\{(\blambda,\bmu) \in \ell^{\infty}_{d^{-p}}(\calX)\times \ell^{\infty}_{d^{-p}}(\calX): \left\langle \br,\blambda\right\rangle +\left\langle \bs,\bmu \right\rangle = W_p^p(\br,\bs)\\
	\lambda_x + \mu_{x'} \leq d^p(x,x') \quad \forall x,x' \in \calX \Big\}
	\end{aligned} 
	\label{eq:dual_set_rs}
	\end{equation}
	and
	\begin{equation}
		\begin{aligned}\calS^*(\br) = \Big\{\blambda \in \ell^{\infty}_{d^{-p}}(\calX): 	\lambda_x - \lambda_{x'} &\leq d^p(x,x') \quad \forall x,x' \in \supp(\br)\Big\},
	\label{eq:dual_set}
	\end{aligned} 
	\end{equation}
	with $\supp(\br) = \set{x \in \calX \colon r_x > 0}$.
For our limiting distributions we define the following (multinomial) covariance structure
	\begin{equation}
	\label{eq:def_sigma}
	\Sigma(\br) = \begin{cases}
	r_x(1-r_x) &\text{ if } x=x',\\
	-r_xr_{x'} &\text{ if } x \neq x'.
	\end{cases}
	\end{equation}

\begin{theorem}
	\label{thm:distrlimit_one}
	Let $(\calX, d)$ be a countable metric space and $\br,\bs \in \calP_p(\calX)$, $p \geq 1$, and $\brh_n$ be generated by i.i.d. samples $X_1, ..., X_n \sim \br$. Furthermore, let $\bG \sim \calN(0, \Sigma(\br))$ be a Gaussian process with $\Sigma$ as defined in \eqref{eq:def_sigma}. Assume \eqref{eq:entropy} for some $x_0 \in \calX$. Then
	\begin{enumerate}
		\item[a)] \label{thm:distrlimit_one_a} 
		\begin{equation}
		\label{eq:one_null}
		n^{\frac{1}{2p}} W_p(\brh_n,\br) \weak \left\lbrace \max_{\blambda \in\calS^*(\br)} \scalprod{\bG,\blambda}\right\rbrace^{\tfrac{1}{p}}, \text{ as } n \to \infty.
		\end{equation}
		
		\item[b)] \label{thm:distrlimit_one_b} In the case where $\br \neq \bs$ it holds for $ n \to \infty$
		\begin{multline}
		\label{eq:one_alternative}
			n^{\frac{1}{2}} (W_p(\brh_n,\bs) - W_p(\br,\bs)) \weak\\ \frac{1}{p} W_p^{1-p}(\br,\bs)\left\lbrace \max_{(\blambda,\bmu)\in\calS^*(\br,\bs)} \scalprod{\bG,\blambda}\right\rbrace.
		\end{multline}
	\end{enumerate}
	
\end{theorem}

Note, that we obtain different scaling rates under equality of measures $\br = \bs$ (null-hypothesis) and the case $\br\neq \bs$ (alternative), which has important statistical consequences. For $\br\neq \bs$ we are in the regime of the standard C.L.T. rate $\sqrt{n}$, but for $\br = \bs$ we get the rate $n^\frac{1}{2p}$, which is strictly slower for $p > 1$. 
\begin{rem}[Degeneracy of limit law]
	We would like to discuss in which settings the limit distribution in \eqref{eq:one_null} is degenerate. \\
	In the case that $\br$ has full support the limit degenerates to a point mass at $0$ if $\calS^*(\br)$ contains only constant elements, i.e., for a $c \in \R$  $\lambda_ x = c$ for all $x \in \calX$. Then, the right hand side in \eqref{eq:one_null} becomes zero. $\calS^*(\br)$ contains only constant elements if and only if the space $\calX$ has no isolated point. \\
	Specifying $\calX$ to be a subset of the real line $\R$ that has no isolated point it follows from Theorem 7.11. in \cite{Bobkov2014} that scaling with $\sqrt{n}$ provides then a non-degenerate limit law. On the other hand, as soon as $\calX \subset \R$ contains an isolated point our rate coincides with the rate given in \cite{Bobkov2014}. 
\end{rem}
\begin{rem} 
	\begin{enumerate}[a)]
		\item Note, that in Theorem \ref{thm:distrlimit_one_b} b) where the measures are not the same the objective function in \eqref{eq:one_alternative} is independent of the second component $\bmu$ of the feasible set $\calS^*(\br,\bs)$. This is due to the fact that in $W_p(\brh_n,\bs)$ the second component is not random.
		\item Observe, that the limit in \eqref{eq:one_alternative} is normally distributed if the set $\calS^*(\br,\bs)$ is a singleton up to a constant shift. In the case of finite $\calX$ conditions for  $\calS^*(\br,\bs)$ to be a singleton up to a constant shift are known \citep{hung_degeneracy_1986,Klee_facets_1968}.
		\item Parallel to our work \cite{del_barrio_central_2017} showed asymptotic normality of the quadratic EWD in general dimensions for the case $\br \neq \bs$. Their results require the measures to have moments of order $4 + \delta$ for some $\delta > 0$ and positive density on their convex support. Their proof relies on a Stein-identity. In the case $\br = \bs$ the limiting distribution is degenerated, in contrast to Thm. \ref{thm:distrlimit_one} a).
		\item The limiting distribution in the case $\br = \bs$ can also be written as
		\[\left\lbrace \max_{\blambda\in\calS^*(\br)} \scalprod{\bG,\blambda}\right\rbrace^{\tfrac{1}{p}} = \left\lbrace \inf_{\bz(\br) \in \ell^\infty_{d^{-p}}(\calX)} W_p(\bG^+ + \bz(\br),\bG^- + \bz(\bz))\right\rbrace^{1/p},\]
		where $ \bG^+$ and $\bG^-$ denotes the (pathwise) decomposition of the Gaussian process $\bG$, such that $\bG = \bG^+ - \bG^-$ and $\bz(\br)$ is related to $\br$ in the sense that $z_x = 0$ for that $x \in \calX$ such that $r_x = 0$. Further, we would like to emphasize that the set of dual solutions $\calS^*(\br)$ is independent of $\br$, if the support of $\br$ is full, i.e.,
		\begin{equation}
		\calS^* = \Big\{\blambda \in \ell^{\infty}_{d^{-p}}(\calX): \lambda_x - \lambda_{x'} \leq d^p(x,x') \quad \forall x,x' \in \calX \Big\}.
		\label{eq:dual_full}
		\end{equation}
		This offers a universal strategy to simulate the limiting distribution on trees independent of $\br$.
		For more details see Appendix \ref{sec:limit_equality}.
		
	\end{enumerate}
\label{rem:theorem1}
\end{rem}

For statistical applications it is also interesting to consider the two sample case, extensions to $k$-samples, $k \geq 2$ being obvious then.
	
\begin{theorem}
	\label{thm:distrlimit_two}
	
	Under the same assumptions as in Thm. \ref{thm:distrlimit_one} and with $\bsh_m$ generated by $Y_1, \ldots, Y_m \overset{iid}{\sim} \bs$, independently of $X_1, \ldots, X_n$ and $\bH \sim \calN(0, \Sigma(\bs))$, which is independent of $\bG$, and the extra assumption that $\bs$ also fulfills \eqref{eq:entropy} the following holds.
	
		\begin{enumerate}
		\item[a)]  Let $\rho_{n,m} = (nm/(n+m))^{1/2}$. If $\br=\bs$ and $\min(n,m) \to \infty$ such that $m/(n+m) \to \alpha \in [0,1]$ we have	\begin{equation}
		\label{eq:two_null}
		\rho^{1/p}_{n,m} W_p(\brh_n,\bsh_m)  \weak \left\lbrace \max_{\blambda \in\calS^*(\br)} \scalprod{\bG,\blambda}\right\rbrace^{\tfrac{1}{p}}.		
		\end{equation}
		
		\item[b)] For $\br \neq \bs$ and $n,m \to \infty$ such that $\min(n,m) \to \infty$ and $m/(n+m) \to \alpha\in [0,1]$ we have
		\begin{equation}
		\label{eq:two_alternative}
		\begin{aligned}
		\rho_{n,m} &(W_p(\brh_n,\bsh_m) - W_p(\br,\bs)) \weak \\
		&\frac{1}{p} W_p^{1-p}(\br,\bs)\left\lbrace \max_{(\blambda,\bmu)\in\calS^*(\br,\bs)} \sqrt{\alpha}\scalprod{\bG,\blambda} + \sqrt{1-\alpha}\left\langle \bH,\bmu\right\rangle \right\rbrace.
		\end{aligned}
		\end{equation}		
	\end{enumerate}
\end{theorem}

\begin{rem}
	In the case of dependent data analogous results to Thm. \ref{thm:distrlimit_one} and \ref{thm:distrlimit_two} will hold, as soon as the weak convergence of the empirical process w.r.t. the $\normfp{\cdot}$-norm is valid. All other steps of the proof remain unchanged.   
\end{rem}

The rest of this subsection is devoted to the proofs of Theorem \ref{thm:distrlimit_one} and Theorem \ref{thm:distrlimit_two}.

\begin{proof}[Proof of Thm. \ref{thm:distrlimit_one} and Thm. \ref{thm:distrlimit_two}]
To prove these two theorems we use the delta method \ref{thm:deltamethod}. Therefore, we need to verify (1.) directional Hadamard differentiability of $W_p(\cdot,\cdot)$ and (2.) weak convergence of $\sqrt{n}(\brh_n - \br)$. We mention that the delta method required here is not standard as the directional Hadamard derivative is not linear (see \cite{romisch_delta_2004}, \cite{shapiro_asymptotic_1991} or \cite{dumbgen_nondifferentiable_1993}). 
\begin{enumerate}

	\item In Appendix \ref{sec:hadamard}, Theorem \ref{thm:Hadamard} directional Hadamard differentiability of $W_p$ is shown  with respect to the $\normfp{\cdot}$-norm \eqref{eq:norm}. \\
	
	\item The weak convergence of the empirical process w.r.t. the $\normfp{\cdot}$-norm is addressed in the following lemma.
	\begin{lem} 
		\label{lem:weak}
		Let $X_1, \ldots, X_n \sim \br$ be i.i.d. taking values in a countable metric space $(\calX, d)$ and let $\brh_n$ be the empirical measure as defined in \eqref{eq:emp_measure}. Then
		\[\sqrt{n}(\brh_n - \br) \xrightarrow{\mathscr{D}} \bG\]
		with respect to the $\normfp{\cdot}$-norm, where $\bG$ is a Gaussian process with mean 0 and covariance structure
		\[\Sigma(\br) = \begin{cases} r_x(1-r_x) & \text{ if } x = x',\\
		-r_xr_{x'} & \text{ if } x \neq x',
		\end{cases}\]
		as given in \eqref{eq:def_sigma} if and only if condition \eqref{eq:entropy} is fulfilled.
	\end{lem} 
	\begin{proof}[Proof of Lemma]
		The weighted $\ell^1$-space $\ell^1_{d^p}$ is according to Prop. 3, \cite{Maurey1973} of cotype 2, hence $\sqrt{n}(\brh_n - \br)$ converges weakly w.r.t. the $\ell^1(d^p)$-norm by Corollary 1 in \cite{Jain1977} if and only if the summability condition \eqref{eq:entropy} is fulfilled.
	\end{proof}
\end{enumerate}
Theorem \ref{thm:distrlimit_one_a} a) is now a straight forward application of the delta method \ref{thm:deltamethod} and the continuous mapping theorem for $f(x) = x^{1/p}$.

For Theorem \ref{thm:distrlimit_one_b} b) we use again the delta method, but this time in combination with the chain rule for directional Hadamard differentiability (Prop. 3.6 (i), \cite{Shapiro1990}).

The proof of Theorem \ref{thm:distrlimit_two} works analogously. Note, that under the assumptions of the theorem it holds $(\br = \bs)$
\begin{multline}
\rho_{n,m} ((\brh_n,\bsh_m) - (\br,\bs)) \\= \left( \sqrt{\frac{m}{n+m}}\sqrt{n}(\brh_n-\br), \sqrt{\frac{n}{n+m}}\sqrt{m}(\bsh_m - \bs)\right)  \\
\weak (\sqrt{\alpha}\bG, \sqrt{1-\alpha}\bG')
\label{eq:two_sample_null}
\end{multline}
with $\bG' \overset{\mathcal{D}}{=} \bG$.
For further explanations see Appendix \ref{sec:limit_equality}.
\end{proof}

\subsection{Examination of the summability condition \eqref{eq:entropy}}
\label{sub:summability}
	According to Lemma \ref{lem:weak} condition \eqref{eq:entropy} is necessary and sufficient for the weak convergence with respect to the $\normfp{\cdot}$-norm defined in \eqref{eq:norm}. As this condition is crucial for our main theorem and we are not aware of a comprehensive discussion, we will provide such in this section.
	
	The following question arises. "If the condition holds for $p$ does it then also hold for all $p' \leq p$?" This is not true in general, but it is true if $\calX$ has no accumulation point (i.e., is discrete in the topological sense).
	\begin{lem}
		Let $\calX$ be a space without any accumulation point with respect to the metric $d$. If condition \eqref{eq:entropy} holds for $p$, then it also holds for all $1 \leq p' \leq p$.
	\end{lem}
\begin{proof}
	Let $\calX$ be a space without an accumulation point, i.e., there exists $\epsilon > 0$ such that $d(x,x') > \epsilon$ for all $x \neq x' \in \calX$. Then, 
	\begin{align*}
	\sum_{x\in\calX} d^p(x_0,x) \sqrt{r_x} &= \epsilon^p \sum_{x \in\calX} \left( \frac{d(x_0,x)}{\epsilon}\right)^p \sqrt{r_x}  \\
	&\geq  \epsilon^p \sum_{x \in\calX} \left( \frac{d(x_0,x)}{\epsilon}\right)^{p'} \sqrt{r_x}\\
	&= \epsilon^{p/{p'}} \sum_{x\in\calX} d^{p'}(x_0,x)\sqrt{r_x}.
	\end{align*}
\end{proof} 
	\paragraph{Exponential families}
	As we will see, condition \eqref{eq:entropy} is fulfilled for many well known distributions including the Poisson distribution, geometric distribution or negative binomial distribution with the euclidean distance as the ground measure $d$ on $\calX = \N$. 
	\begin{theorem}
		\label{thm:sum}
		Let $(\calP_{\boeta})_{\boeta}$ be an s-dimensional standard exponential family (SEF) (see \cite{Lehmann1998}, Sec. 1.5) of the form
		\begin{equation}
		\label{eq:sef}
		r_x^{\boeta} = h_x \exp\left(\sum_{i= 1}^{s} \eta_iT_x^i - A(\boeta)\right).
		\end{equation}
		
		The summability condition \eqref{eq:entropy} is fulfilled if $(\calP_{\boeta})_{\boeta}$ satisfies
		\begin{itemize}
			\item[1.)]  $h_x \geq 1$ for all $x \in \calX$,
			\item[2.)] the natural parameter space $\calN$ is closed with respect to multiplication with $\frac{1}{2}$, i.e., $\sum_{x\in\calX} r^{\boeta}_x < \infty \Rightarrow \sum_{x\in\calX} r^{\boeta/2}_x < \infty$,
			\item[3.)] the $p$-th moment w.r.t. the metric $d$ on $\calX$ exists, i.e., $\sum_{x\in\calX} d^p(x,x_0)r^{\boeta}_x < \infty$ for some arbitrary, but fixed $x_0 \in \calX$.
		\end{itemize}
	\end{theorem}
	
	\begin{proof}
		For the SEF in \eqref{eq:sef} condition \eqref{eq:entropy} reads
		\begin{align}\sum_{x\in \calX} d^p(x_0,x) &\sqrt{\exp\left(\sum_{i=1}^{s}\eta_iT^i_x  - A(\boeta)\right) h_x} \nonumber\\
		&= \frac{1}{\sqrt{\lambda(\boeta)}}\sum_{x\in\calX}d^p(x_0,x) \exp\left(\tfrac{1}{2}\sum_{i=1}^{s}\eta_iT_x^i\right)\sqrt{h_x} \\
		&\leq \frac{\lambda(\tfrac{1}{2}\boeta)}{\sqrt{\lambda(\boeta)}}\sum_{x\in\calX}d^p(x_0,x) \exp\left(\tfrac{1}{2}\sum_{i=1}^{s}\eta_iT_x^i\right)h_x < \infty \nonumber,		
		\end{align}
		where $\lambda(\boeta)$ denotes the Laplace transform. The first inequality is due to the fact that $h_x \geq 1$ for all $x\in\calX$ and the second is a result of the facts that the natural parameter space is closed with respect to multiplication with $\tfrac{1}{2}$ and that the $p$-th moment w.r.t. $d$ exist.
	\end{proof}

	The following examples show, that all three conditions in Theorem \ref{thm:sum} are necessary.
	 
	\begin{example}
	 Let $\calX$ be the countable metric space $\calX = \left\lbrace \frac{1}{k}\right\rbrace_{k \in \N}$ and let $\br$ be the measure with probability mass function $$r_{1/k} = \frac{1}{\zeta(\eta)}\frac{1}{k^\eta}$$ with respect to the counting measure. Here, $\zeta(\eta)$ denotes the Riemann zeta function. This is an SEF with natural parameter $\eta$, natural statistic $-\log(k)$ and natural parameter space $\calN = (1,\infty).$
		We choose the euclidean distance as the distance $d$ on our space $\calX$ and set $x_0 = 1$. It holds
		\[\sum_{k=1}^{\infty} \abs{1 -\frac{1}{k}}^p\frac{1}{\zeta(\eta)}\frac{1}{k^{\eta}} \leq \sum_{k=1}^{\infty} \frac{1}{\zeta(\eta)}\frac{1}{k^{\eta}} = 1 <\infty \quad \forall  \eta\in \calN\]
		and hence all moments exist for all $\eta$ in the natural parameter space. Furthermore, $h_{1/k} \equiv 1$. 
		However, the natural parameter space is not closed with respect to multiplication with $\frac{1}{2}$ and therefore,
		\[\sum_{k=1}^{\infty} \abs{1 -\tfrac{1}{k}}^p\frac{1}{\zeta(\eta)}\frac{1}{k^{\eta/2}} \geq \frac{1}{2^p} \sum_{k=2}^{\infty} \frac{1}{\sqrt{\zeta(\eta)}}\frac{1}{k^{\eta/2}} = \infty \quad \forall  \eta\in (1,2],\]
		i.e., condition \eqref{eq:entropy} is not fulfilled.
	\end{example}
	The next example shows, that we cannot omit condition 1.) in Thm. \ref{thm:sum}.
	\begin{example}
		Consider $\calX = \N$ with the metric $d(k,l) = \sqrt{\abs{k!-l!}}$. The family of Poisson distributions constitute an SEF with natural parameter space $\calN = (-\infty,\infty)$ which satisfies condition 2.) in Thm. \ref{thm:sum}, i.e., closed with respect to multiplication with $\tfrac{1}{2}$. The first moment with respect to this metric exists and $h_k < 1$ for all $k \geq 2$. Condition \eqref{eq:entropy} for $p = 1$ with $x_0 = 0$ reads 
		\[\sum_{k=1}^{\infty} \sqrt{k!} \sqrt{\frac{\eta^k}{k!}\exp(-\eta) } = \sum_{k=1}^{\infty} \eta^{k/2}\exp(-\eta/2) = \infty\]
		for all $\eta > 1$, i.e., the summability condition \eqref{eq:entropy} is not fulfilled.
	\end{example}	
	If the $p$-th moment does not exist, it is clear that condition \eqref{eq:entropy} cannot be fulfilled as $\sqrt{x} \geq x$ for $x \in [0,1]$.\\

	\subsection{Approximation of continuous distributions}
	In this section we investigate to what extent we can approximate continuous measures by its discretization such that condition \eqref{eq:entropy} remains valid. Let $\calX = \left(\frac{k}{M}\right)_{k\in\Z}$ with $M \in \N$ be a discretization of $\R$ and $X$ a real-valued random variable with c.d.f. $F$ which is continuous and has a Lebesgue density $f$. We take $d$ to be the euclidean distance and $x_0 = 0$. For $k \in \Z$ we define 
	\begin{equation}
	r_k := F\left(\frac{k+1}{M}\right) - F\left(\frac{k}{M}\right).
	\label{eq:bin}
	\end{equation}
	Now, \eqref{eq:entropy} can be estimated as follows.
	\begin{align*}
	&\sum_{k=-\infty}^{\infty} \abs{\frac{k}{M}}^p \sqrt{F\left(\frac{k+1}{M}\right) - F\left(\frac{k}{M}\right)}\\
	& =\sum_{k=-\infty}^{\infty} \abs{\frac{k}{M}}^p \frac{1}{\sqrt{M}} \sqrt{M \int_{k/M}^{(k+1)/M}f(x) dx}\\
	&\geq  \sum_{k=-\infty}^{\infty} \abs{\frac{k}{M}}^p \sqrt{M}  \int_{k/M}^{(k+1)/M} \sqrt{f(x)} dx\\
	& \geq \sqrt{M} \sum_{k=-\infty}^{\infty} \frac{1}{2^p}\int_{k/M}^{(k+1)/M} \abs{x}^p\sqrt{f(x)} dx \\
	&= \sqrt{M} \frac{1}{2^p}\int_{\R} \abs{x}^p \sqrt{f(x)}dx,
	\end{align*}
	 where the first inequality is due to Jensen's inequality. As the r.h.s. tends to infinity with rate $\sqrt{M}$ as $M \to \infty$, condition \eqref{eq:entropy} does not hold in the limit. Hence, in general our method of proof cannot be extended in an obvious way to continuous measures.
%
	 	
\paragraph{The one-dimensional case $D =1$}	 
For the rest of this Section we consider $\calX = \R$ and want to put condition \eqref{eq:entropy} in relation to the condition \citep{delBarrio_clt_1999}
 \begin{equation}
 \int_{-\infty}^{\infty} \sqrt{F(t)(1-F(t))} dt < \infty,
 \label{eq:Barrio}
 \end{equation}
 where $F(t)$ denotes the cumulative distribution function, which is sufficient and necessary for the empirical 1-Wasserstein distance on $\R$ to satisfy a limit law (see also Corollary 1 in \cite{Jain1977} in a more general context). \\  
 Condition \eqref{eq:entropy} is stronger than \eqref{eq:Barrio} as the following shows. 
 Let $\calX$ be a countable subset of $\R$ and index the elements $x_i$ for $i \in \Z$ such that they are ordered. Furthermore, let $d(x,y) = \abs{x-y}$ be the euclidean distance on $\calX$. For any measure $\br$ with cumulative distribution function $F$ on $\calX$ it holds
 \begin{equation*}
 \begin{aligned}
 &\int_{-\infty}^{\infty} \sqrt{F(t)(1-F(t))} dt\\
 & = \sum_{k\in \Z} d(x_k,x_{k+1})\sqrt{\sum_{j \leq k} r_j}\sqrt{\sum_{j>k}r_j}\\
 &\leq \sum_{k=0}^{\infty} d(x_k,x_{k+1}) \sqrt{\sum_{j > k} r_j} + \sum_{k=-\infty}^{-1} d(x_k,x_{k+1}) \sqrt{\sum_{j \leq k} r_j} \\
 & \leq \sum_{k=0}^{\infty} d(x_k,x_{k+1}) \sum_{j > k}\sqrt{r_j} + \sum_{k=-\infty}^{-1} d(x_k,x_{k+1}) \sum_{j \leq k}\sqrt{ r_j} \\
 & =  \sum_{k=0}^{\infty} d(x_0,x_{k})\sqrt{r_k} + \sum_{k=-\infty}^{-1} d(x_0,x_{k}) \sqrt{r_k}.
 \end{aligned} 	
 \end{equation*}
Hence, if condition \eqref{eq:entropy} holds, \eqref{eq:Barrio} is also fulfilled. However, the conditions are not equivalent as the following example shows.
\begin{example}
	Let $\calX = \N$ and $d(x,y) = \abs{x-y}$ the euclidean distance and $\br$ a power-law, i.e., $r_n = \frac{1}{\zeta(s)}\frac{1}{n^s}$, where $\zeta(s)$ is the Riemann zeta function. In this case \eqref{eq:Barrio} reads
	\begin{equation*}
	\begin{aligned}
	&\int_{-\infty}^{\infty} \sqrt{F(t)(1-F(t))} dt =\frac{1}{\zeta(s)} \sum_{k = 1}^{\infty} \sqrt{\sum_{j = 1}^{k} \frac{1}{j^s} \sum_{j = k+1}^{\infty} \frac{1}{j^s}}\\
	&\leq \frac{1}{\zeta(s)} \sum_{k = 1}^{\infty} \sqrt{\sum_{j = k}^{\infty} \frac{1}{j^s}}
 \lesssim \frac{1}{\zeta(s)} \sum_{k = 1}^{\infty} \sqrt{\frac{s}{k^{s- 1}}}
	\end{aligned}
	\end{equation*}
	and this is finite if and only if $s > 3$. 
	On the other hand, condition \eqref{eq:entropy} reads as 
	\[\sum_{k =1}^{\infty} (k-1) \sqrt{\frac{1}{\zeta(s)}\frac{1}{k^s}} \leq \frac{1}{\sqrt{\zeta(s)}} \sum_{k= 1}^{\infty} \frac{1}{k^{s/2 -1}}.\]
	This is finite if and only if $s > 4$. Hence, condition \eqref{eq:Barrio} is fulfilled for $s \in (3,4]$, but not \eqref{eq:entropy}.
\end{example}	
For $p = 2$ in dimension $D = 1$ there is no such easy condition anymore in the case of continuous measures, see \cite{Barrio2005}. Already for the normal distribution one needs to subtract a term that tends sufficiently fast to infinity to get a distributional limit (which was originally proven by \cite{deWet1972}). Nevertheless, for a fixed discretization of the normal distribution via binning as in \eqref{eq:bin} condition \eqref{eq:entropy} is fulfilled and Theorems \ref{thm:distrlimit_one} and \ref{thm:distrlimit_two} are valid.

\subsection{Bounded diameter of $\calX$}
\label{sub:bounded_dia}
For $\calX$ with bounded diameter further simplifications can be obtained.\\
First and most important, we do not need to introduce the spaces $\ell^1_{d^p}(\calX)$ and its dual $\ell^\infty_{d^{-p}}(\calX)$ in this case. This is due to the fact, that as the diameter of the space with respect to the metric $d$ is bounded all moments of probability measures on this space exist. Hence, we do not need to restrict to probability measures that have finite $p$-th moment to guarantee that the linear program \eqref{eq:linearprog} defining the Wasserstein distance has a finite value. Thus, we can operator on $\calP(\calX)$ which is a subset of $\ell^1(\calX)$. This simplifies the summability condition \eqref{eq:entropy} to 
\[\sum_{x \in \calX} \sqrt{r_x} < \infty \] as we get directional Hadamard differentiability with respect to the $\norm{\cdot}_1$-norm.

%% file: section/tree.tex
\section{Limiting Distribution for Tree Metrics}
\label{sec:tree}

\subsection{Explicit limits}
\label{sub:explicit_tree}
In this subsection we give an explicit expression for the limiting distribution in  
\eqref{eq:one_null} and \eqref{eq:two_null} in the
case $\br = \bs$  with full support (otherwise see Rem. \ref{rem:non_full_support_tree}) when the metric is
generated by a weighted tree. This extends Thm. 5 in \cite{sommerfeld_inference_2018} for finite spaces to countable spaces $\calX$. In the following we recall their notation.

Assume that the metric structure on the countable space $\calX$ is given by a weighted tree, that is, an
undirected connected graph $\calT = (\calX, E)$
with vertices  $\calX$ and edges $E \subset \calX\times \calX$ that
contains no cycles. We assume the edges to be weighted by a function
\[
w:E \rightarrow \R_{+}.
\]
Without imposing any further restriction on $\calT$, we assume it to be rooted at
$\rootT (\calT)\in \calX$, say. Then, for $x\in \calX$ and $x\neq\rootT(\calT)$  we may define
$\parent(x)\in \calX$ as the
immediate neighbor of $x$ in the unique path connecting $x$ and $\rootT(\calT)$. 
We set $\parent(\rootT(\calT))=\rootT(\calT)$.
We also define $\children(x)$ as the set of vertices $x'\in \calX$ such that there
exists a sequence $x' = x_1, \dots , x_n = x \in \calX$ with $\parent(x_j) =
x_{j+1}$ for $j=1,\dots,n-1$. Note that with this definition $x\in\children(x)$. Furthermore, observe that $\children(x)$ can consist of countably many elements, but the path joining $x$ and $x' \in \children(x)$ is still finite as explained below.  

For $x, x'\in \calX$ let $e_1,\dots,e_n\in
E$ be the unique path in $\calT$ joining $x$ and $x'$, then the length of this
path, 
\[
d_\calT(x,x') = \sum_{j=1}^n w(e_j),
\]
defines a metric $d_\calT$ on $\calX$. This metric is well defined, since the unique path joining $x$ and $x'$ is finite as we show in the following. Let $A_0 = \set{x \in \calX : x = \rootT(\calT)}$ and $A_k = \set{x \in \calX: \parent(x) \in A_{k-1}}$ for $k \in \N$. By the definition of the $A_k$, these sets are disjoint and it follows $\bigcup_{k = 0}^{\infty} A_k = \calX$. Now let $x,x' \in \calX$, then there exist $k_1$ and $k_2$ such that $x \in A_{k_1}$ and $x' \in A_{k_2}$. Then, there is a sequence of $k_1 + k_2 +1$ vertices connecting $x$ and $x'$. Hence, the unique path joining $x$ and $x'$ has at most $k_1 + k_2$ edges.

Additionally,  
define 
\[
(S_\calT \bu)_x = \sum_{x'\in\children(x)} u_{x'}
\]
and 
\begin{equation}
Z_{\calT,p}(\bu) = \left\{\sum_{
	x \in \calX} |(S_\calT \bu)_x| d_\calT(x,\parent(x))^p\right\}^{\frac{1}{p}}
\label{eq:def_ZT}
\end{equation}
for $\bu\in\R^{\calX}$ and we set w.l.o.g. $x_0 = \rootT(\calT)$.

The main result of this section is the following.
\begin{theorem}
	\label{THM:TREES}
	Let $\br\in \mathcal{P}_p(\calX)$, defining a probability distribution on $\calX$ that fulfils condition \eqref{eq:entropy} and
	let the empirical measures  $\brh_n$ and
	$\bsh_m$
	be generated by independent random variables $X_1,\dots, X_n $ and $Y_1 ,
	\dots Y_m$, respectively, all drawn from
	$\br = \bs$. 
	
	Then, with a Gaussian vector $\bm{G}\sim\mathcal{N}(0, \Sigma(\br))$ with $\Sigma(\br)$ as defined in
	\eqref{eq:def_sigma} we have the following.
	\begin{enumerate}[a)]
		\item \textbf{(One sample)} As $n\rightarrow\infty$, 
		\begin{equation}
		n^{\frac{1}{2p}} W_p(\brh_n, \br) \weak Z_{\calT,p}(\bG)
		\label{eq:weak_conv_trees_one}
		\end{equation}
		\item \textbf{(Two sample)} If  $n\wedge m\rightarrow\infty$ and
		$n/(n + m)\rightarrow\alpha \in [0,1]$ we have 
		\begin{equation}
		\left( \frac{nm}{n+m} \right)^{\frac{1}{2p}} W_p(\brh_n, \bsh_m)
		\weak Z_{\calT,p}(\bG).
		\label{eq:weak_conv_trees_two}
		\end{equation}
	\end{enumerate}
\end{theorem}
A rigorous proof of Thm. \ref{THM:TREES} is given in Appendix \ref{sec:proof_tree}.

The same result was derived in \cite{sommerfeld_inference_2018} for finite spaces. For $\calX$ countable we require a different technique of proof. Simplifying the set of dual solutions in the same way, the second step of rewriting the target function with a summation and difference operator does not work in the case of measures with countable support, since the inner product of the operators applied to the parameters is no longer well defined. For this setting we need to  introduce a new basis in $\ell^1_{d^p}(\calX)$ and for each element $\bmu \in \ell^1_{d^p}(\calX)$ a sequence which has only finitely many non-zeros that converges to $\bmu$ in order to obtain an upper bound on the optimal value. Then, we define a feasible solution for which this upper bound is attained.

\begin{rem}
	In case that the support is not full we can generate a weighted tree for the support points in the following way. If $x$ is not in the support of $\br$ we delete $x$ and connect $\parent(x)$ to all nodes in the set $A_{+1} (x) = \set{x' \in \calX \colon \parent(x') = x}$ with edges that have the length of the sum of the edge joining $x$ and $\parent(x)$ and the edge joining $x' \in A_{+1}$ and $x$. Then, we can use the same arguments as in the case of full support to derive the explicit limit on the restricted tree. This is an upper bound of the limiting distribution on the full tree with non full support. See Figure \ref{fig:tree_supp} for an illustration.
	\begin{figure}
		 \subfigure[full tree]{
			\begin{tikzpicture}
			\draw[red] (0,0) 
			-- (1.5,1.5)
			node [circle, fill = black, scale = 0.5, label = above:{\textcolor{black}{$root$}}]{}
			node at (0.5,1) {$w_1$};
			\draw (1.5,1.5)
			-- (3,0)
			node [circle, fill = black, scale = 0.5]{}
			-- (3.5,-1.5)
			node [circle, fill = black, scale = 0.5]{}
			-- (3.5, -2.5)
			(3.5,-2.7) node {$\vdots$};
			\draw (3,0) -- (2.5, - 1.5)
			node [circle, fill = black, scale = 0.5]{}
			-- (2.5, -2.5)
			(2.5,-2.7) node {$\vdots$};
			\draw[red] (0,0) -- (-0.5,-1.5)
			node [circle, fill = black, scale = 0.5]{}
			node at (-0.5,-0.7) {$w_2$};
			\draw (-0.5,-1.5)
			-- (-0.5, -2.5)
			(-0.5,-2.7) node {$\vdots$};
			\draw[red] (0,0) -- (0.5,-1.5)
			node [circle, fill = black, scale = 0.5]{}
			node[shape = circle, draw, fill = white, scale = 0.5] at (0,0) {}
			node at (0.5,-0.7) {$w_3$};
			\draw (0.5,-1.5)
			-- (0.5, -2.5)
			(0.5,-2.7) node {$\vdots$};
		\end{tikzpicture}
	}
		 \subfigure[tree reduced to support]{
		 	\begin{tikzpicture}
		 	\draw[red] (-0.5,-1.5) 
		 	-- (1.5,1.5)
		 	node [circle, fill = black, scale = 0.5, label = above:{\textcolor{black}{$root$}}]{}
		 	node at (-0.3,0) {$w_1 + w_2$};
		 	\draw (1.5,1.5)
		 	-- (3,0)
		 	node [circle, fill = black, scale = 0.5]{}
		 	-- (3.5,-1.5)
		 	node [circle, fill = black, scale = 0.5]{}
		 	-- (3.5, -2.5)
		 	(3.5,-2.7) node {$\vdots$};
		 	\draw (3,0) -- (2.5, - 1.5)
		 	node [circle, fill = black, scale = 0.5]{}
		 	-- (2.5, -2.5)
		 	(2.5,-2.7) node {$\vdots$};
		 	\draw
		 	node[circle, fill = black, scale = 0.5] at (-0.5,-1.5){};
		 	\draw (-0.5,-1.5) -- (-0.5, -2.5)
		 	node at (-0.5,-2.7) {$\vdots$};
		 	\draw (0.5,-1.5) -- (0.5, -2.5)
		 	node at (0.5,-2.7) {$\vdots$}
		 	node[circle, fill = black, scale = 0.5] at (0.5, -1.5) {} ;
		 	\draw[red] (1.5,1.5) -- (0.5,-1.5)
		 	node at (1.7,0) {$w_1 + w_3$};
		 	\end{tikzpicture}

}
			\label{fig:tree_supp}
			\caption{Schematic for the reduction of $\calX$ to the support of $\br$. Solid circles indicate support points, hollow circles elements which are not in the support.}
	\end{figure}
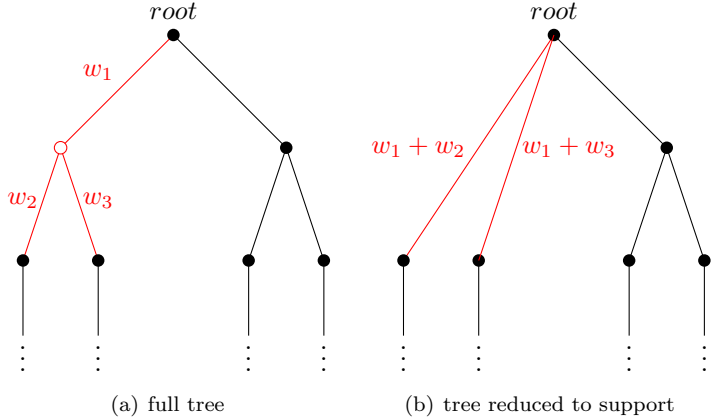
\label{rem:non_full_support_tree}
\end{rem}

\subsection{Distributional Bound for the Limiting Distribution}
\label{sub:bound}
In this section we use
the explicit formula on the r.h.s. of \eqref{eq:weak_conv_trees_one} for the case of  tree metrics 
to stochastically bound the limiting distribution on a general space $\calX$ which is not a tree. 

This is based on the following simple observation: Let $\calT$ be a spanning tree of
$\calX$ and $d_\calT$ the tree metric generated by $\calT$ and the weights $(x,x')\mapsto
d(x,x')$ as described in Section \ref{sub:explicit_tree}. Then for any $x,x'\in\calX$ we
have $d(x,x')\leq d_\calT(x,x')$. Let $\calS_{\calT}^*$ denote the set defined in
$\eqref{eq:dual_set}$ with the metric $d_\calT$ instead of $d$. Then $\calS^*\subset
\calS_{\calT}^*$ and hence
\[
\max_{\blambda \in\calS^*} \scalprod{\bv,\blambda} \leq \max_{\blambda\in\calS_{\calT}^*}
\scalprod{\bv,\blambda}
\]
for all $\bv\in\ell^1_{d^p}(\calX)$.
It follows that
\begin{equation}
\max_{\blambda \in\calS^*} \scalprod{\bv,\blambda} \leq Z_{\calT,p}(\bv)
\label{eq:upper_bound}
\end{equation}
for all $\bv\in\ell^1_{d^p}(\calX)$ and
this  proves the following main result of this subsection, which is stated for the case, when $\br$ and $\bs$ are both estimated from data. The one-sample case is analogous. 
\begin{theorem}
	\label{thm:convex_strategy}
	Let $\br, \bs\in\calP_p(\calX)$, assume that $\br, \bs$ fulfill condition \eqref{eq:entropy} and let $\brh_n$, $\bsh_m$ be generated by i.i.d.
	$X_1, \dots, X_n\sim \br$ and $Y_1,\dots,Y_m\sim \bs$, respectively. Let
	further $\calT$ be a spanning tree of $\calX$. Then,
	if $\br = \bs$ we have, 
	as $n$ and
	$m$ approach infinity such that $n\wedge m\rightarrow\infty$ and
	$n/(n + m)\rightarrow\alpha$, that 
	\begin{equation}
	\begin{split}
	\limsup_{n,m\rightarrow\infty}  P&\left[ \left( \frac{nm}{n + m} \right)^{1/2p} W_p(\brh_n, \bsh_m) \geq z
	\right] \leq
	P\left[ Z_{\calT,p}(\bG) \geq z \right],  
	\end{split}
	\label{eq:convex_strategy}
	\end{equation}
	where $\bG \sim\mathcal{N}(0,\Sigma(\br))$  with
	$\Sigma(\br)$ as
	defined in \eqref{eq:def_sigma}.
\end{theorem}
\begin{rem}
		While the stochastic bound of the limiting distribution $Z_{\calT,p}$ is very fast
		to compute as it is explicitly given, the Wasserstein distance $W_p(\brh_n, \bsh_m)$ in \eqref{eq:convex_strategy} is a
		computational bottleneck. 
		Classical general-purpose approaches, e.g., the simplex algorithm
		\citep{Luenberger2008} for general linear programs or the auction algorithm for
		network flow problems \citep{bertsekas_auction_1992, bertsekas_auction_2009}
		were found to scale rather poorly  to very large problems
		such as image retrieval \citep{rubner_earth_2000}.
		
		Attempts to solve this problem include specialized algorithms
		\citep{Gottschlich2014} and approaches leveraging additional geometric structure
		of the data \citep{ling_efficient_2007, schmitzer_sparse_2016}.
		However, many practical problems still fall outside the scope of
		these methods \citep{schrieber_dotmark_2017}, prompting the development of numerous surrogate quantities which
		mimic properties of optimal transport distances and are amenable to efficient
		computation. Examples include \cite{pele_fast_2009,shirdhonkar_approximate_2008,
		bonneel_sliced_2015} and the particularly successful entropically regularized
		transport distances \citep{Cuturi2013a,solomon_convolutional_2015}.
		 
\end{rem}

In the next section we will discuss how to approximate the countable space $\calX$ by a finite collection of points. Note, that the distributional bound in Thm. \ref{thm:convex_strategy} also holds on any finite collection of points. For a simulation study regarding this upper bound see \cite{tameling_computational_2018}.

%% file: section/rates.tex
\section{Computational strategies for simulating the limit laws}
\label{sec:rates}

If we want to simulate the limiting distributions in Thm. \ref{thm:distrlimit_one} and \ref{thm:distrlimit_two} we need to restrict to a finite number $N$ of points, i.e., we choose a subset $I$ of $\calX$ such that $\#I = N$. Let $\br \in \calP_p(\calX)$ with full support (see Remark \ref{rem:support} for the general case), satisfying \eqref{eq:entropy}. For $\bG \sim \calN(0,\Sigma(\br))$, we define $\bG^I = (G^I)_x = G_x\one_{\left\lbrace x \in I\right\rbrace }$. Then, an upper bound for the difference between the exact limiting distribution and the limiting distribution on the finite set $I$ in the one sample case for $\br = \bs$ is given as (see \eqref{eq:upper_bound})

\begin{equation}
\begin{aligned}
\abs{\max_ {\blambda \in \calS^*} \scalprod{\bG^I,\blambda} - \max_ {\blambda \in \calS^*} \scalprod{\bG, \blambda}} &\leq
\max_ {\blambda \in \calS^*} \abs{ \scalprod{\bG^I,\blambda} - \scalprod{\bG, \blambda}} \\
&\leq  \max_ {\blambda \in \calS_{\calT}^*} \abs{\scalprod{\bG^I - \bG, \blambda}}\\
& = \max\Big\{ \max_ {\blambda \in \calS_{\calT}^*} \scalprod{\bG^I - \bG, \blambda},  \\ &\hspace{10ex}\max_ {\blambda \in \calS_{\calT}^*} \scalprod{\bG - \bG^I, \blambda} \Big\} \\
 &= \sum_{
	x \in \calX} |(S_\calT (\bG^I - \bG))_x| d_\calT(x,\parent(x))^p\\
&= \sum_{x \notin I} \abs{G_x} d_\calT(x,\rootT(\calT))^p .
\end{aligned}
\label{eq:finite_approx1}
\end{equation}

For the last equality one needs to construct the tree as follows: Choose $I$ such that $x_0$ from condition \eqref{eq:entropy} is an element of $I$ and choose $x_0$ to be the root of the tree and let all other elements of $\calX$ be direct children of the root, i.e., $\children(x) = x$ for all $x \neq \rootT(\calT) \in \calX$. The upper bound can be made stochastically arbitrarily small as 
\begin{equation}
\mean\left[ \sum_{x \notin I} \abs{G_x} d_\calT(x,\rootT(\calT))^p\right] \leq   \sum_{x \notin I} d_\calT(x,\rootT(\calT))^p \sqrt{r_x(1-r_x)},
\label{eq:finite_approx} 
\end{equation}
where we used Hölder's inequality and the definition of $\Sigma(\br)$. As the root was chosen to be $x_0$, the sum above is finite as $\br$ fulfills condition \eqref{eq:entropy} and becomes arbitrarily small for $I$ large enough. Hence, \eqref{eq:finite_approx} details that the speed of approximation by $\bG^I$ depends on the decay of $\br$ and suggests to choose $I$ such that most of the mass of $\br$ is concentrated on it. \\

\begin{rem}
	In case that the support of $\br$ is not full, we have to optimize over the set $\calS^*(\br)$ given in \eqref{eq:dual_set}. In this case we can derive the same upper bound as in \eqref{eq:finite_approx1} with the only change that we sum over all $x \in \supp(\br)$ in the second last line of \eqref{eq:finite_approx1} and that our set $I$ has to be a subset of the support of $\br$.
	\label{rem:support}
\end{rem}

The computation of $\max_ {\blambda \in \calS^*} \scalprod{\bG^I,\blambda}$ is a linear program with $N^2$ constraints and $N$ variables. General purpose network flow algorithms such as the auction algorithm, Orlin's algorithm or general purpose LP solvers are required for the computation of this linear problem. These algorithms have at least cubic worst case complexity \citep{bertsekas_new_1981, orlin_faster_1993} and quadratic memory requirement and its average runtime is much worse than $\mathcal{O}(N^2)$ empirically \citep{Gottschlich2014}. 
This renders a naive Monte-Carlo approach to obtain quantiles computational infeasible for large $N$. In the following subsections we therefore discuss possibilities to make the computation of the limit more accessible.

\subsection{Thresholded Wasserstein distance}
\label{sub:convex}

Following \cite{pele_fast_2009} we define for a thresholding parameter $t\geq 0$  the thresholded metric
\begin{equation}
  d_t(x,x') = \min\left\{ d(x,x'), t \right\}.
  \label{eq:def_dt}
\end{equation}
Then, $d_t$ is again a metric. Let
$\Wt_p(\br, \bs)$ be the Wasserstein distance with respect to 
$d_t$. Since $d_t(x,x')\leq d(x,x')$ for all $x,x'\in\calX$ we have that
$\Wt_p(\br, \bs) \leq W_p(\br, \bs)$ for all $\br, \bs\in\calP (\calX)$ and all
$t\geq 0$. 

\begin{theorem}
	\label{thm:threshold}
	The limiting distribution from Thm. \ref{thm:distrlimit_one_a} with the thresholded ground distance $d_t$ instead of $d$ can be computed in $\mathcal{O}(N^2\log N)$ time with $\mathcal{O}(N)$ memory requirement, if each point in $\calX$ has $\mathcal{O}(1)$ neighbors with distance smaller or equal to $t$. The limiting distribution can be calculated as the optimal value of the following network flow problem: 
	\begin{equation}
	\label{eq:threshold_limdist}
	\begin{aligned}
	\min_{\bw \in \R^{\calX \times \calX}_+} - \sum_{x,x' \in \calX}d_t^p(x,x')w_{x,x'} \\
	\text{subject to }\sum_{\tilde{x}\in \calX, \tilde{x} \neq x} w_{\tilde{x},x} - \sum_{x' \in \calX, x' \neq x} w_{x,x'} = G_x,
	\end{aligned}
	\end{equation}
	
	where $\bG = (G_x)_{x \in  \calX}$ is a Gaussian process with mean zero and covariance structure as defined in \eqref{eq:def_sigma}.
\end{theorem}

\begin{proof}
	We take a finite approximation $\br_N$ of $\br$ and reduce our space $\calX$ to the support of $\br_N$ which should be exactly $N$ points. If we take the thresholded distance as the ground distance similar as in Theorem \ref{thm:distrlimit_one} we obtain the limiting distribution as
	\[ \left\lbrace \max_{\blambda \in \calS_t^*}  \scalprod{\bG,\blambda}\right\rbrace ^{1/p}, \]
	where now $\calS_t^* = \set{\blambda  \in \R^N : \lambda_x - \lambda_{x'} \leq d_t^p(x,x')}$. The $p$-th power of the limiting distribution is again a finite dimensional linear program and since there is strong duality in this case, it is equivalent to solve \eqref{eq:threshold_limdist}.
	As the linear program \eqref{eq:threshold_limdist} is a network flow problem, we can redirect all edges with length $t$ through a
	virtual node without changing the optimal value. From the assumption that each point has $\mathcal{O}(1)$ neighbors with distance not equal to $t$, we can deduce that the number of edges ($N^2$ in the original problem) is reduced to $\mathcal{O}(N)$. According to \cite{pele_fast_2009} the new linear program with the virtual node can be solved in $\mathcal{O}(N^2\log N)$ time with $\mathcal{O}(N)$ memory requirement.
\end{proof}

\begin{rem}
	
	\begin{itemize}
		\item[a)] The resulting network-flow problem can be tackled with existing efficient
		solvers \citep{bertsekas_auction_1992} or commercial solvers like $CPLEX$ (https://www.ibm.com/jm-en/marketplace/ibm-ilog-cplex) which exploit the network structure.
		\item[b)] For the distributional bound \eqref{eq:convex_strategy} one can also use the thresholded Wasserstein distance $\Wt_p$ instead of $W_p$ to be computational more efficient. A large threshold $t$ will result in a better
		approximation of the true Wasserstein distance, but will also require more computation time. 
	\end{itemize}
	
\end{rem}

\subsection{Regular Grids}
\label{sub:regular_grids_asymp}
In this section we are going to derive an explicit formula for the distributional bound from Section \ref{sub:bound}, when the support of $\br$ is a regular grid of $L^D$ points in the unite hypercube $[0,1]^D$. Here, $D$ is a positive integer and $L$ a power of two. In this case a spanning tree can be constructed from a dyadic partition.
The general case is analogous, but more cumbersome.  
For $0\leq l \leq l_{\max}$ with 
\[
l_{\max} = \log_2 L
\] 
let $P_l$ be the natural partition of $\supp(\br)$ into $2^{Dl}$ squares of each
$L^D / 2^{Dl}$ points.
\begin{theorem}
	Under the assumptions described above, \eqref{eq:convex_strategy} reads
	\begin{equation}
	Z_{\calT,p}(\bu) =\left\{ \sum_{l=0}^{l_{\max}} D^{p/2} 2^{-p(l+1)} \sum_{F\in
		P_l} |S_F \bu|  \right\}^{1/p}. 
	\label{eq:bound_grid}
	\end{equation}
	This expression can be evaluated efficiently (in $L^D\log_2 L$ operations) and used with Theorem
	\ref{thm:convex_strategy} to obtain a stochastic bound of the limiting distribution on regular grids. 
\end{theorem}
 
\begin{proof}
	Define $\supp(\br)'$ by adding to $\supp(\br)$ all center-points of sets in $P_l$ for $0\leq l
	<l_{\max}$. We identify center points of $P_{l_{\max}}$ 
	with the points in $\supp(\br)$. A tree with vertices $\supp(\br)'$ can now be build using the
	inclusion relation of the sets $\left\{ P_l \right\}_{0\leq l \leq l_{\max}}$
	as ancestry relation. More precisely, the leaves of the tree are the
	points of $\supp(\br)$ and the parent of the center point of $F\in P_l$ is the center
	point of the unique set in
	$P_{l-1}$ that contains $F$. 
	
	If we use the Euclidean metric to define
	the distance between neighboring vertices we get
	\[
	d_\calT(x,\parent(x)) = \frac{\sqrt{D}2^{-l}}{2},
	\]
	if $x\in P_l$.
	
	A measure $\br$ naturally extends to a measure on $\supp(\br)'$ if we give zero
	mass to all inner vertices. We also denote this measure by $\br$. Then, if
	$x\in\supp(\br)'$ is the center point of the set
	$F\in P_l$ for some $0\leq l \leq l_{\max}$, we have that $(S_\calT \br)_x = S_F
	\br$ where $S_F\br = \sum_{x\in F} r_x$. Inserting this two formulas into \eqref{eq:convex_strategy} yields \eqref{eq:bound_grid}.
\end{proof}

%% file: section/SMS.tex
\section{Application: Single-Marker Switching Microscopy} \quad\\
\label{sec:appl_SMS}
Single Marker Switching (SMS) Microscopy
\citep{betzig_imaging_2006,rust_subdiffractionlimit_2006,
egner_fluorescence_2007, heilemann_subdiffractionresolution_2008,folling_fluorescence_2008} 
is a living cell fluorescence microscopy technique
in which fluorescent markers which are tagged to a protein structure in the probe are stochastically switched from a no-signal
giving (off) state into a signal-giving (on) state. A marker in the on state
emits a bunch of  photons some of which are detected on a detector before
it is either switched off or bleached. From the photons registered on
the detector, the position of the marker (and hence of the protein) can be determined. The final image is
assembled from all observed individual positions recorded in a sequence of time intervals (frames) in a position histogram, typically a
pixel grid. 

SMS microscopy is based on the principle that at any given time
only a very small number of markers are in the on state. As the probability of switching from the off to the on state is small for each
individual marker and they remain in the on state only for a very short
time (1-100ms). This allows SMS microscopy to resolve features below the
diffraction barrier that limits conventional far-field microscopy
(see \cite{hell_farfield_2007} for a survey) because with overwhelming probability at most one
marker within a diffraction limited spot is in the on state \citep{aspelmeier_modern_2015}. At the same time
this requires quite long acquisition times (1min-1h) to guarantee sufficient sampling of the probe. As a consequence, if the
probe moves during the acquisition, the final image will be blurred.

Correcting for this drift and thus improving image quality is an area of active
research
\citep{geisler_drift_2012,deschout_precisely_2014,hartmann_drift_2014}. In order to investigate the validity of such a drift correction method we introduce a test of the Wasserstein distance between the image obtained from the first half of the recording time and the second half. This test is based on the distributional upper bound of the limiting distribution which was developed in Section \ref{sub:bound} in combination with a lower bound of the Wasserstein distance \citep{pele_fast_2009}. In fact, there is no standard method for problems of this kind and we argue that the (thresholded) Wasserstein distance is particular useful in such a situation as the specimen moves between the frames without loss of mass, hence the drift induces a transport structure between successive frames. In the following we compare the distribution from the first half of frames with the distribution from the second half scaled with the sample sizes (as in \eqref{eq:weak_conv_trees_two}). We reject the hypothesis that the distributions from the first and the second half are the same, if our test statistic is larger than the $1-\alpha$ quantile of the distributional bound of the limiting distribution in \eqref{eq:convex_strategy}. If we have statistical evidence that the thresholded Wasserstein distance is not zero, we can also conclude that there is a significant difference in the Wasserstein distance itself.

\paragraph{Statistical Model}
It is common to assume the bursts  of photons registered on the detector as
independent realizations
of a random variable with a density that is proportional to the density of
markers in the probe \citep{aspelmeier_modern_2015}. As it is expected that the probe drifts during the acquisition this density
will vary over time. In particular, the positions registered at the beginning of
the observation will follow a different distribution than those observed at the
end. 
\paragraph{Data and Results}
We consider an SMS image of  a tubulin structure presented in
\cite{hartmann_drift_2014} to assess their drift correction method. This image is recorded in 40.000 single frames over a total recording time of 10 minutes (i.e., 15 ms per frame). We compare
the aggregated sample collected during the first $50\%$ ($\hat{=}$ 20.000 frames) of the total
observation time with the aggregated sample obtained in the last $50\%$ on a $256\times
256$ grid  for both the
original uncorrected values and for the values where the drift correction of
\cite{hartmann_drift_2014} was applied. Heat maps of these
four samples are shown in the left hand side of Figure \ref{fig:SMS} (no correction) and Figure \ref{fig:SMS_Corr} (corrected), respectively. 

\begin{figure}
  \centering
  \makebox[\textwidth][c]{
  \begin{minipage}{0.6\textwidth}
  	  \begin{tabular}{c}
  		\hspace{1mm}\subfigure{\includegraphics[width=\textwidth]{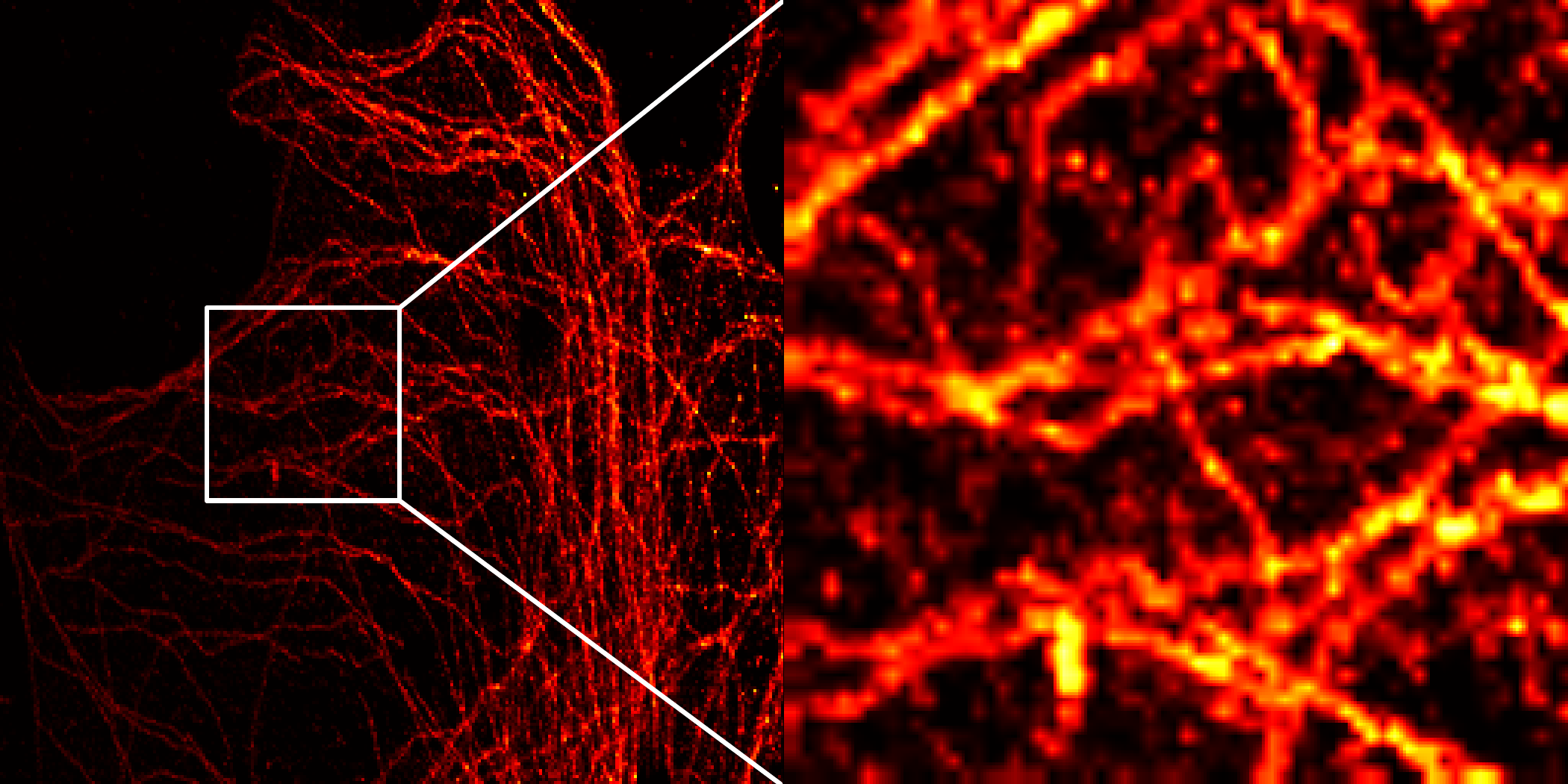}} \vspace{-1em}\\
  		\subfigure{\includegraphics[width=\textwidth]{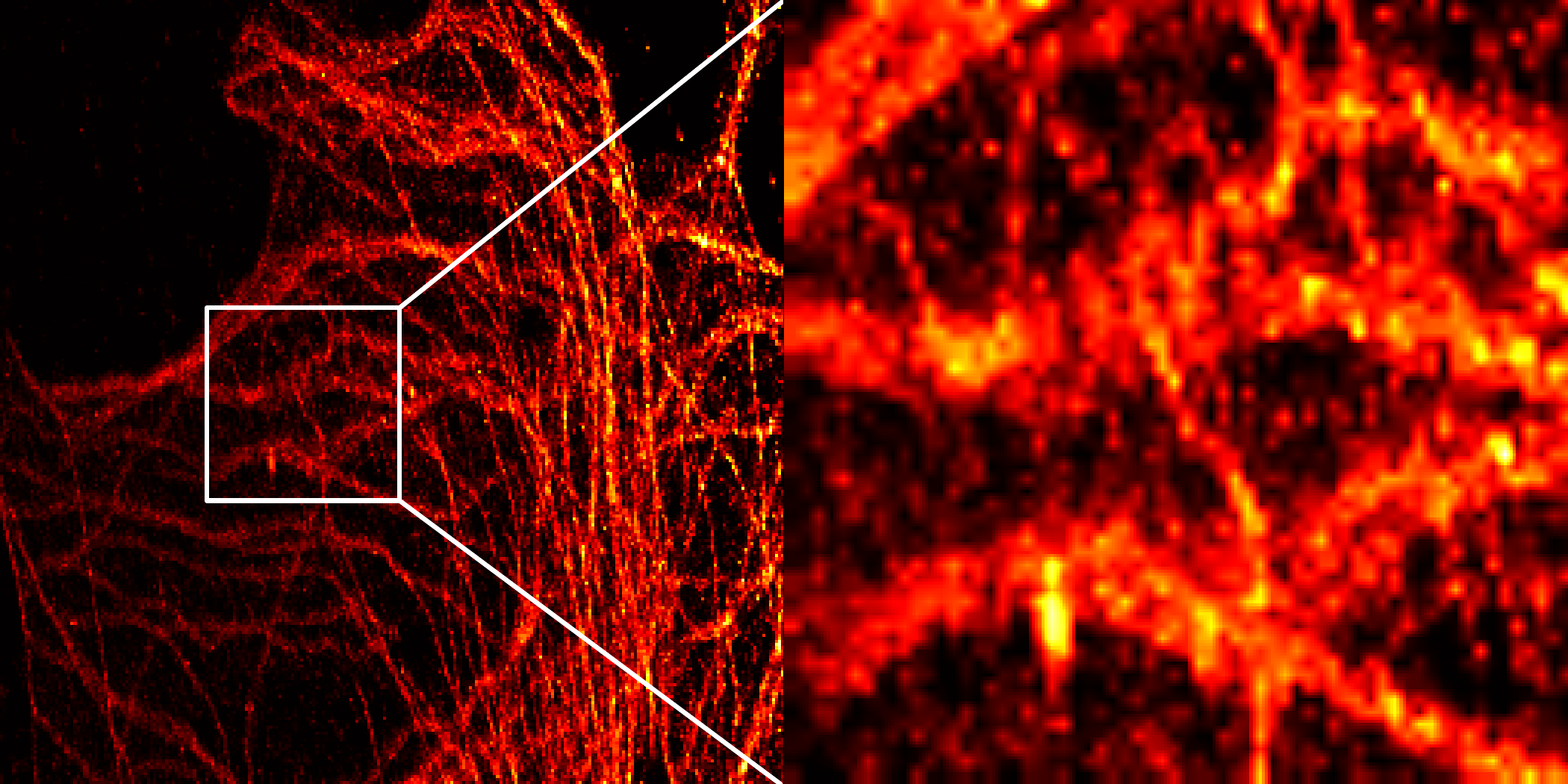}} 
  	\end{tabular}
  \end{minipage}
\hspace{1em}
  \begin{minipage}{0.4\textwidth}
	 \subfigure{\includegraphics[width=\textwidth]{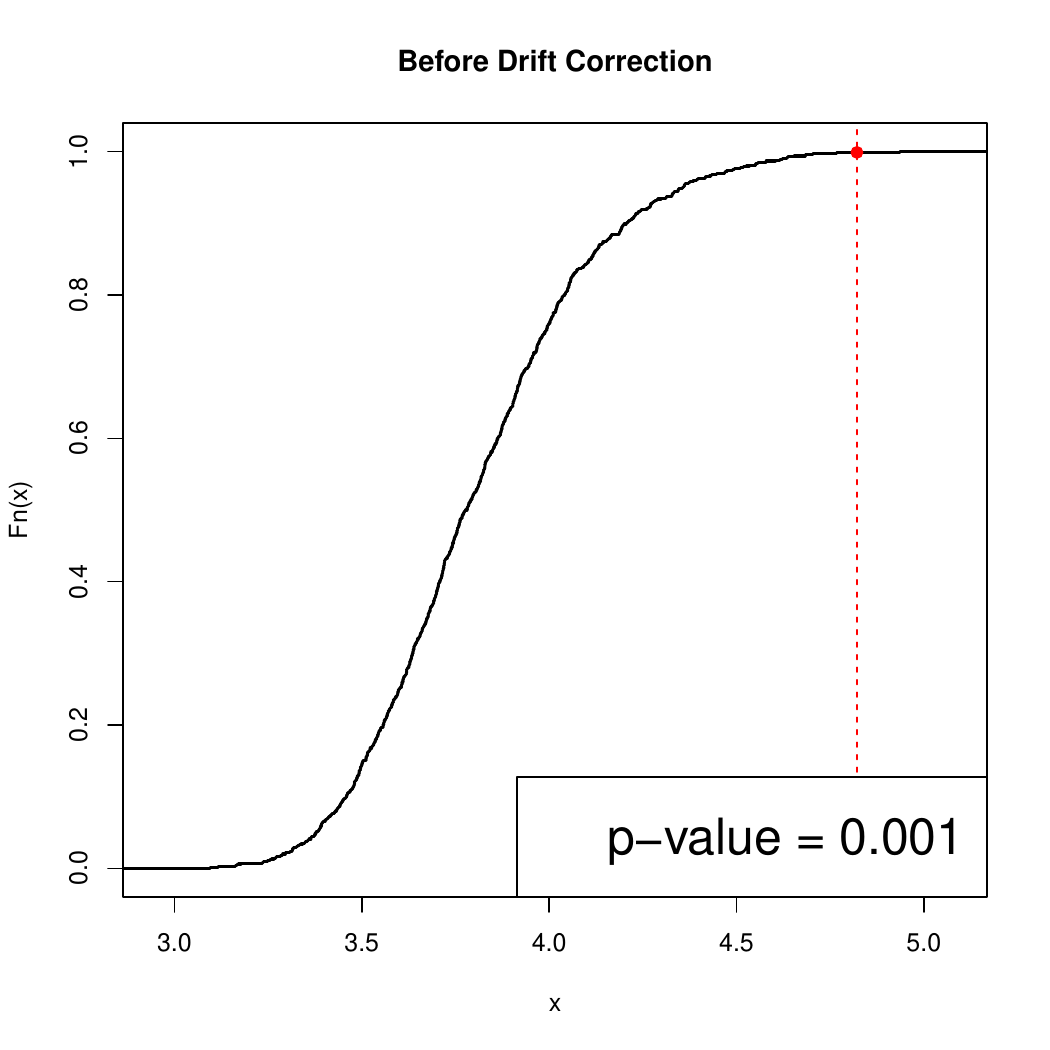}}
  \end{minipage}}

  \caption{Left: Aggregated samples of the first (first row) and the last (second row)
    $50\%$ of the observation time as heat maps of relative frequency without correction for the drift of the probe.
    Magnifications of a small area are shown to highlight the blurring of the
    picture. 
    Right: Empirical distribution function of a sample from the upper bound (tree approximation) of the limiting distribution. The red dot (line) indicates the scaled thresholded Wasserstein distance for $t = 6/256$.}
  \label{fig:SMS}
\end{figure}  

\begin{figure}
  \centering
\makebox[\textwidth][c]{
	\begin{minipage}{0.6\textwidth}
		\begin{tabular}{c}
			\hspace{1mm}\subfigure{\includegraphics[width=\textwidth]{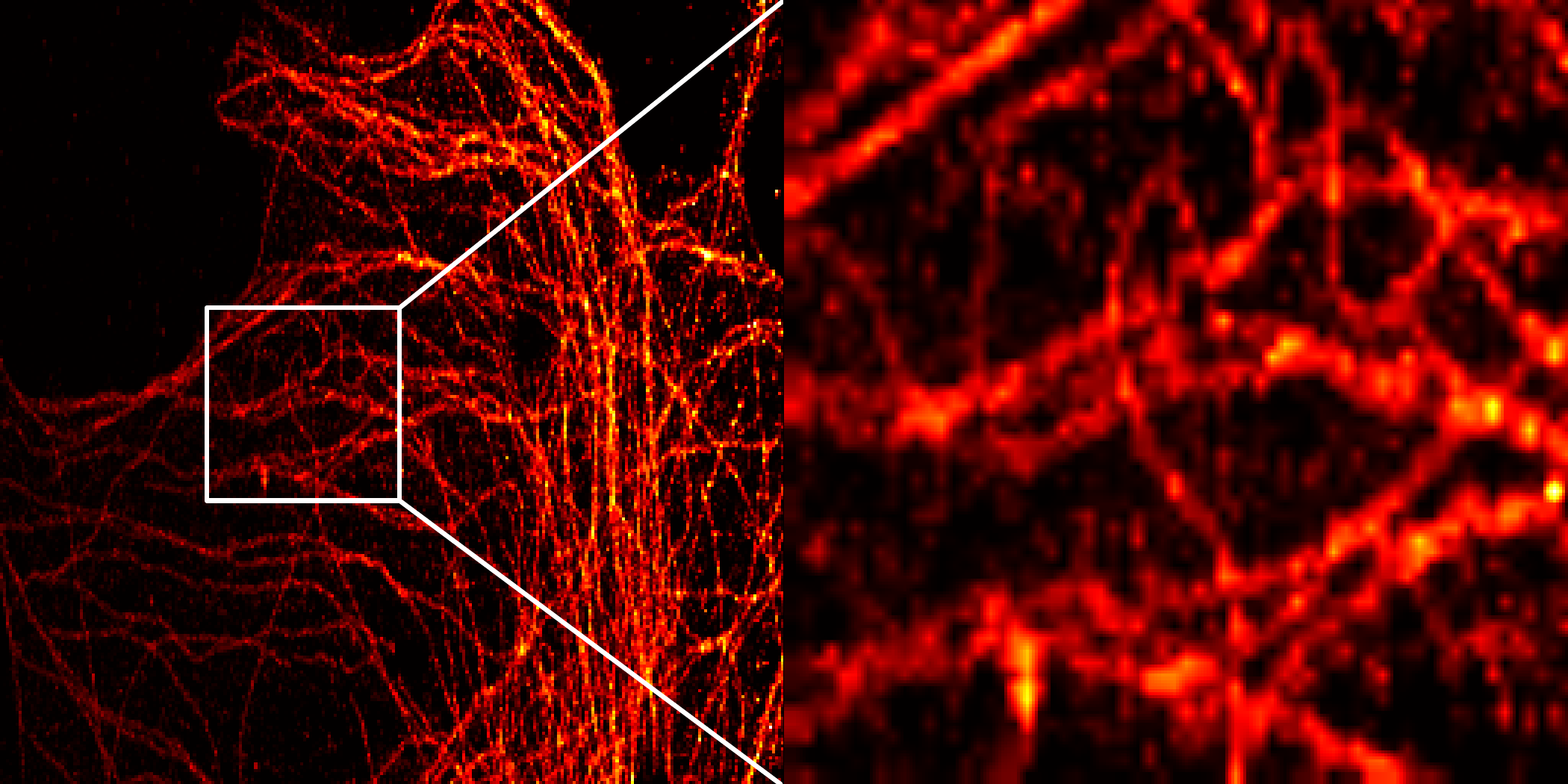}} \vspace{-1em}\\
			\subfigure{\includegraphics[width=\textwidth]{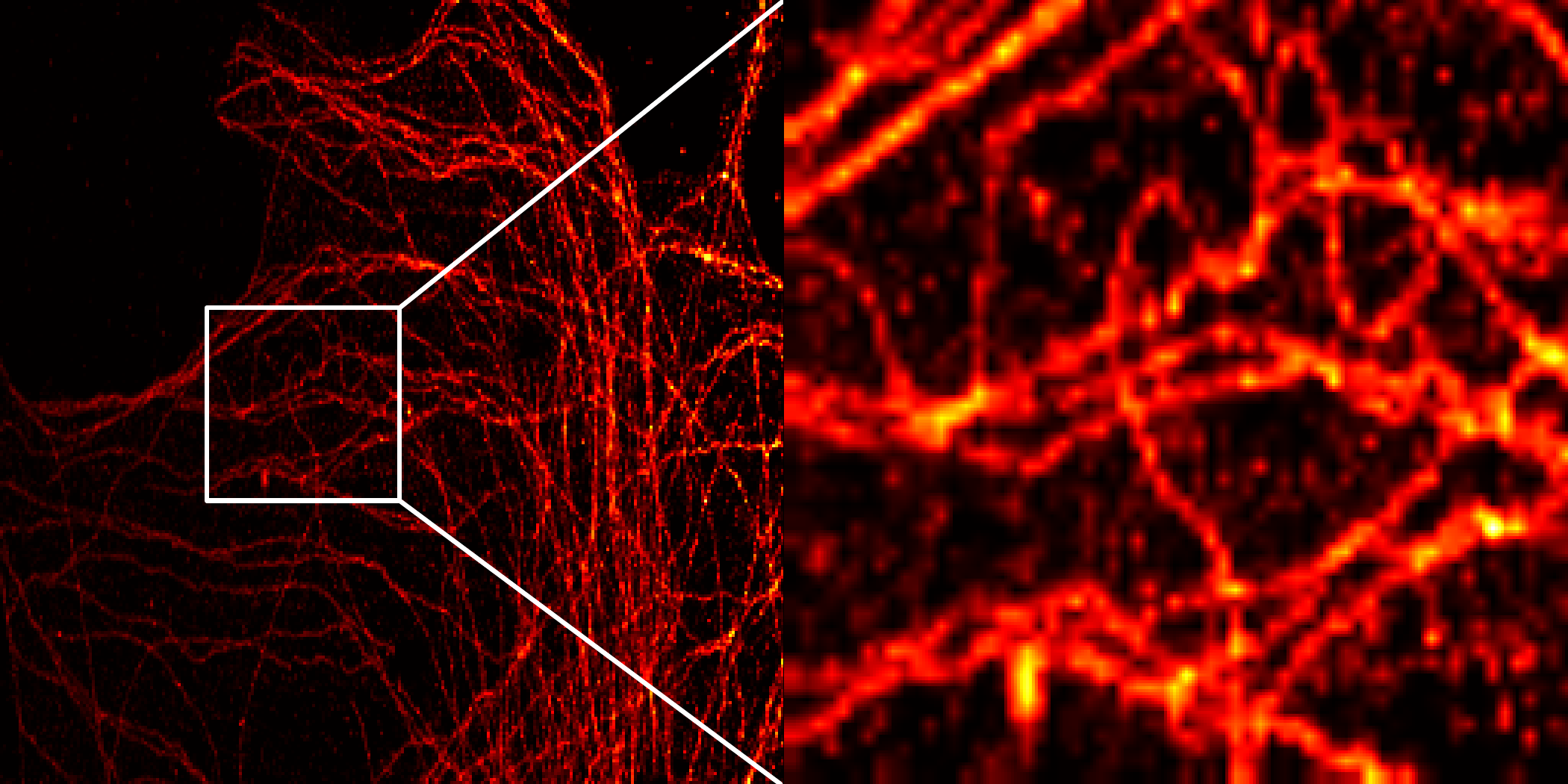}} 
		\end{tabular}
	\end{minipage}
	\hspace{1em}
	\begin{minipage}{0.4\textwidth}
		\subfigure{\includegraphics[width=\textwidth]{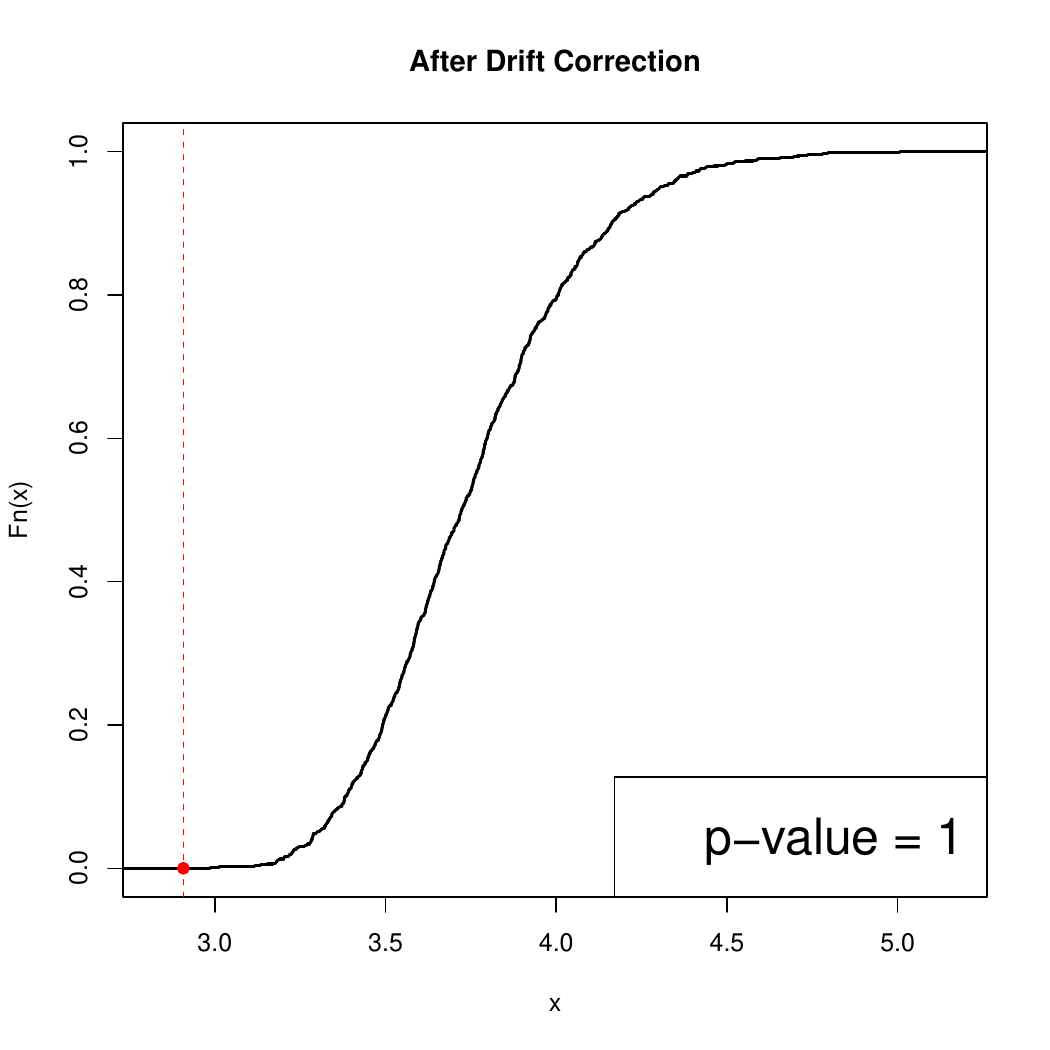}}
\end{minipage}}
		\caption{Left: Aggregated samples of the first (first row) and the last (second row)
			$50\%$ of the observation time as heat maps of relative frequency with correction for the drift of the probe.
			Magnifications of a small area are shown to highlight the drift correction of the
			picture. 
			Right: Empirical distribution function of a sample from the upper bound (tree approximation) of the limiting distribution. The red dot (line) indicates the scaled thresholded Wasserstein distance after drift correction for $t = 6/256$. The difference between the fist and the second $50\%$ is no longer significant.}
		\label{fig:SMS_Corr}
\end{figure}

The question we will address is: "To what extend has the drift been properly removed by the drift correction?"
In addition, from the application of the thresholded Wasserstein distance for different thresholds we expect to obtain detailed understanding for which scales the drift has been removed. As \cite{hartmann_drift_2014} have corrected with a global drift function one might expect that on small spatial scales not all effects have been removed.  

We compute the thresholded Wasserstein distance $\Wt_1$ between the
two pairs of
samples as described in Section \ref{sub:convex} with different thresholds $t \in \{2,3,\ldots, 14 \}/256$. We
compare these values with a sample from the stochastic upper bound for the limiting
distribution on regular grids obtained as described in Section \ref{sub:regular_grids_asymp}.
This allows us to obtain a test for the null hypothesis 'no difference' based on
Theorem \ref{thm:convex_strategy}. To visualize the outcomes of theses tests for different thresholds $t$ we have plotted the corresponding p-values in Figure \ref{fig:pvals}. The red line indicates the magnitude of the drift over the total recording time. As the magnitude is approximately $6/256$, we plot in the right hand side of Figure \ref{fig:SMS} and Figure \ref{fig:SMS_Corr} the empirical distribution functions of the upper bound \eqref{eq:convex_strategy} and indicate the value of the test-statistic for $t = 6/256$ with a red dot without the drift correction and with the correction, respectively. 

\begin{figure}
	\centering
	\includegraphics[width = 0.5\textwidth]{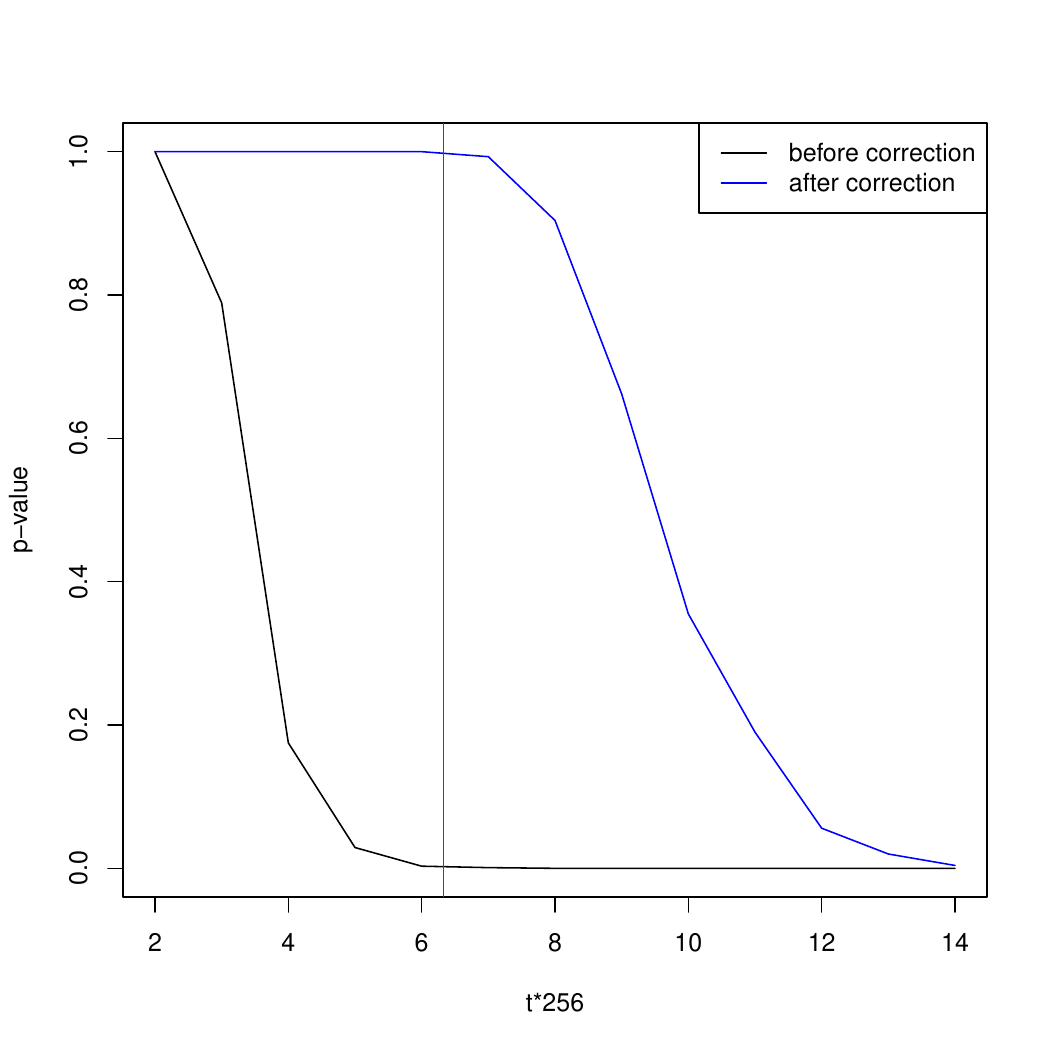}
	\caption{P-values for the null hypothesis 'no difference' for different thresholds $t$ before and after the drift correction. The red line indicates the magnitude of the total drift.}
	\label{fig:pvals}
\end{figure}

As shown in Figure \ref{fig:pvals} the differences caused by the drift of the
probe are recognized as highly statistically significant ($p \leq 0.05$) for thresholds larger than $t = 4/256$. After the drift correction
method is applied, the difference is no
longer significant for thresholds smaller than $t = 14/256$. The estimated shift during the first and the second 50\% of the observations is three pixels in x-direction and one pixel in y-direction. That shows that the significant difference that is detected when comparing the images without drift correction for $t \in \{5,6,7,8,9,10\}/ 256$ is caused in fact by the drift. The fact that there is still a significant difference for large thresholds ($t \geq 14$) in the corrected pictures suggests further intrinsic and local inhomogeneous motion of the specimen or non-polynomial drift that is not captured by the drift model used in \cite{hartmann_drift_2014} and bleaching effects of fluorescent markers.  

In summary, this example demonstrates that our strategy of combining a lower
bound for the Wasserstein distance with a stochastic bound of the limiting
distribution is capable of detecting subtle differences in a large $N$ setting.

%% file: section/appendix.tex
\appendix
\section{Proofs}
\subsection{Hadamard directional differentiability}
\label{sec:hadamard}
In this section we follow mainly \cite{shapiro_asymptotic_1991} and \cite{romisch_delta_2004}. Let $\calU$ and $\calY$ be normed spaces. 
\begin{definition}[cf. \cite{shapiro_asymptotic_1991}, \cite{romisch_delta_2004}] \quad \\
	\begin{itemize}
		\item[a)] \textbf{Hadamard directional differentiability} \\
		A mapping $f\colon D_f \subset \calU \to \calY$ is said to be Hadamard directionally differentiable at $u \in \calU$ if for any sequence $h_n$ that converges to $h$ and any sequence $t_n \searrow 0$ such that $u +t_nh_n \in D_f$ for all $n$ the limit
		\begin{equation}
		f_u'(h) = \lim_{n \to \infty} \frac{f(u + t_nh_n) - f(u)}{t_n}
		\label{eq:directHadamard}
		\end{equation}
		exist.
		\item[b)] \textbf{Hadamard directional differentiability tangentially to a set}
		Let $K$ be a subset of $\calU$, $f$ is directionally differentiable tangentially to $K$ in the sense of Hadamard at $u$ if the limit \eqref{eq:directHadamard} exists for all sequences $h_n$ that converge to $h$ of the form $h_n = t_n^{-1}(k_n - u)$ where $k_n \in K$ and $t_n \searrow 0$.
		This derivative is defined on the contingent (Bouligand) cone to $K$ at $u$
		\[T_K(u) = \set{h \in \calU: h = \lim_{n\to\infty} t_n^{-1}(k_n - u), k_n \in K, t_n \searrow0}. \]
	\end{itemize}
	
\end{definition}
Note that this derivative is not required to be linear in $h$, but it is still positively homogeneous. Moreover, the directional Hadamard derivative $f_u'(\cdot)$ is continuous if $u$ is an interior point of $D_f$ \citep{romisch_delta_2004}.\\
 
The delta method for mappings that are directionally Hadamard differentiable tangentially to a set reads as follows:
\begin{theorem}[\cite{romisch_delta_2004}, Theorem 1]
	\label{thm:deltamethod}
	Let $K$ be a subset of $\calU$, $f\colon K \to \calY$ a mapping and assume that the following two conditions are satisfied: 
	\begin{enumerate}[i)]
		\item The mapping $f$ is Hadamard directionally differentiable at $u \in K$ tangentially to $K$ with derivative $f_u'(\cdot)\colon T_{K}(u) \to \calY$.
		\item For each $n$, $X_n\colon \Omega_n \to K$ are maps such that $a_n(X_n - u) \weak X$ for some sequence $a_n \to +\infty$ and some random element $X$ that takes values in $T_K(u)$. 
	\end{enumerate}
	Then we have $a_n(f(X_n) - f(u)) \xrightarrow{\mathscr{D}} f_u'(X).$ 
\end{theorem}

\paragraph[Hadamard directional differentiability of the Wasserstein distance on countable metric spaces]{Hadamard directional differentiability of the  Wasserstein distance on countable metric spaces}

For $\br,\bs \in \calP_p(\calX)$ the $p$-th power of the $p$-th Wasserstein distance is the optimal value of an infinite dimensional linear program. We use this fact to verify that the $p$-th power of the Wasserstein distance \eqref{eq:wasser} on the countable metric spaces $\calX$ is directionally Hadamard differentiable with methods of sensitivity analysis of optimal values.\\

The $p$-th power of the Wasserstein distance on countable metric spaces is the optimal value of the following infinite dimensional linear program
\begin{equation}
\label{eq:linearprog}
\begin{aligned}
\min_{\bw \in \ell^1_{d^p}(\calX\times \calX)} &\sum_{x,x' \in \calX} d^p(x,x')w_{x,x'}\\
\text{subject to } & \sum_{x'\in\calX} w_{x,x'} = r_x \quad \forall x \in \calX, \\  &\sum_{x\in\calX} w_{x,x'} = s_{x'}, \quad \forall x'\in \calX,\\
&w_{x,x'} \geq 0, \quad \forall x,x'\in \calX.
\end{aligned}
\end{equation}
\begin{theorem}
	\label{thm:Hadamard}
	$W_p^p$ as a map from $(\calP_p(\calX) \times \calP_p(\calX), \normfp{\cdot})$ to $\R$, $(\br,\bs) \mapsto W_p^p(\br,\bs)$ is Hadamard directionally differentiable tangentially to $\calP_p(\calX) \times \calP_p(\calX)$. The contingent cone on which the derivative is defined is given by  
	\[\calD(\br,\bs) = \calD(\br) \times \calD(\bs)\] with 
	\[\calD(\br) \defeq \set{\bd \in \ell^1_{d^p}(\calX)\backslash\{0\} \colon \sum_{x \in \calX} {d}_x = 0, ~ d_x\in [-r_x,1-r_x]} \] and the directional derivative is as follows
	\begin{equation}
	\label{eq:derivative}
	(\bd_1,\bd_2) \mapsto \sup_{(\blambda,\bmu) \in \calS^*(\br,\bs)} -(\scalprod{\blambda,\bd_1} + \scalprod{\bmu,\bd_2}),
	\end{equation}
	where $\calS^*(\br,\bs)$ is set of optimal solutions of the dual problem which is defined in \eqref{eq:dual_set_rs}. 
	
\end{theorem}

\begin{proof}
	
	We start the proof with stating the considered functions and the spaces on which they are defined. The \textit{objective function} of the linear program that determines the $p$-th power of the $p$-th Wasserstein distance is given as $f\colon \ell^1_{d^p}(\calX \times \calX) \to \R, \bw \mapsto \sum_{x,x' \in \calX} d^p(x,x')w_{x,x'}$. The constraints are encoded by the \textit{constraint function} $C\colon \ell^1_{d^p}(\calX \times \calX) \times \ell^1_{d^p}(\calX) \times \ell^1_{d^p}(\calX) \to \ell^1_{d^p}(\calX \times \calX) \times \ell^1_{d^p}(\calX) \times \ell^1_{d^p}(\calX)$ with
	\begin{equation}
	C(\bw,(\br,\bs)) = \begin{pmatrix}
	\bw\\
	\Sigma_1 \bw - \br \\
	\Sigma_2 \bw - \bs\\
	\end{pmatrix},
	\label{eq:constraint_fun}
	\end{equation}
	here $\Sigma_1,\Sigma_2\colon \ell^1_{d^p}(\calX \times \calX) \to \ell^1_{d^p}(\calX)$ are the summation operators over the first and the second component, i.e., $\Sigma_1\bw = \sum_{x' \in \calX} w_{x,x'}$ and $\Sigma_2\bw = \sum_{x \in \calX} w_{x,x'}$. Furthermore, we need the closed convex set $K= \ell^1_{d^p}(\calX \times \calX)_+ \times \set{\bnull}\times \set{\bnull}$ were $\ell^1_{d^p}(\calX \times \calX)_+$ are the elements in $\ell^1_{d^p}(\calX \times \calX)$ that have only non-negative entries. With these definitions the $p$-th power of the $p$-th Wasserstein distance is the optimal value of the following abstract parametrized optimization problem:
	\begin{equation}
	\min_{\bw \in \ell^1_{d^p}(\calX \times \calX)} f(\bw) \text{ s.t. } C(\bw,(\br,\bs)) \in K
	\end{equation}\\
	We will use Theorem 4.24 from \cite{bonnans_perturbation_2000}. To this end, we need to check the following three conditions. 

	\begin{itemize}
			
		\item[(i.)] \textit{Convexity and existence of optimal solution} \\
		Problem \eqref{eq:linearprog} is obviously convex as it is a linear program with linear constraints. Note that the definition of a convex problem (Def. 2.163) in \cite{bonnans_perturbation_2000} is slightly different from the usual definition of a convex program as they require convexity of the constraint function \eqref{eq:constraint_fun} \textit{with respect to $-K$}. This condition can be shown by easy calculations for our problem. \\
		The set of primal optimal solutions, $\calS(\br,\bs)$, is according to Thm. 4.1 in \cite{villani_optimal_2008} non empty.

		\item[(ii.)] \textit{Directional regularity}\\
		Set for some direction $(\bd_1,\bd_2)\in \calD(\br,\bs) \subset \ell^1_{d^p}(\calX)\times \ell^1_{d^p}(\calX)$ 
			\[\bar{C}(\bw,t) = (\bw,\bw^T\one - \br - t\bd_1, \bw\one - \bs - t\bd_2, t).\]
		The directional regularity condition is fulfilled at $\bw_0$ in a direction $(\bd_1,\bd_2)$ if Robinson's constraint qualification is satisfied at the point $(\bw_0,0)$ for the mapping $\bar{C}(\bw,t)$ with respect to the set $K \times \R_+$ \citep[Def. 4.8]{bonnans_perturbation_2000}.
		According to Thm. 4.9 in \cite{bonnans_perturbation_2000} the following condition is necessary and sufficient for the directional regularity constraint to hold:
		\[\boldsymbol{0} \in \inter{C(\bw_0,(\br,\bs)) + DC(\bw,(\br,\bs))(\ell^1_{d^p}(\calX \times \calX), \R_+(\bd_1,\bd_2)) - K} ,\]
		where $ \R_+(\bd_1,\bd_2) = \set{t(\bd_1,\bd_2), t \geq 0}$.
		We are going to show that the directional regularity condition in a direction $(\bd_1,\bd_2)\in \calD(\br,\bs)$ 
		holds for all primal optimal solutions $\bw_0 \in \calS(\br,\bs)$.\\
		For a primal optimal solution $\bw_0$ it is 
		\[C(\bw_0,(\br,\bs)) = (\bw_0,\bnull,\bnull).\]
		Since $C(\bw,(\br,\bs))$ is linear in $(\bw, (\br,\bs))$  and bounded with respect to the product norm on the space $\ell^1_{d^p}(\calX \times \calX) \times \ell^1_{d^p}(\calX) \times  \ell^1_{d^p}(\calX)$ it holds that\\
		$DC(\bw_0,(\br,\bs))(\ell^1_{d^p}(\calX \times \calX),\R_+(\bd_1,\bd_2)) = (\bw,\Sigma_1\bw - t\bd_1,\Sigma_2\bw- t\bd_2)$ for $t \geq 0$ and the directional regularity condition reads
		\[\bnull \in \inter{(\bw_0,\bnull,\bnull) + (\bw,\Sigma_1\bw -t\bd_1 , \Sigma_2\bw -t\bd_2) - K}.\]
		This set is just $\ell^1_{d^p}(\calX \times \calX) \times \ell^1_{d^p}(\calX) \times  \ell^1_{d^p}(\calX)$ as $\bw \in \ell^1_{d^p}(\calX \times \calX)$ and hence the directional regularity constraint is fulfilled.\\
		
		\item[(iii.)] \textit{Stability of primal optimal solution}\\
		We aim to verify that for perturbed measures of the form $\br_n = \br + t_n\bd_1 + o(t_n)$  and $\bs_n = \bs + t_n \bd_2 + o(t_n)$ with $t_n \searrow 0$, $\br, \bs \in \calP_p(\calX)$, $\bd_1 \in \calD(\br)$ and $\bd_2 \in \calD(s)$ there exist a sequence of primal optimal solutions $\bw_n$ that converges to the primal optimal solution $\bw_0$ of the unperturbed problem. For $n$ large enough $t_n \leq 1$, hence we can assume without loss of generality that $t_n \leq 1$ for all n. In this case $\br_n$ and $\bs_n$ are probability measure with existing $p$-th moment, i.e. elements of $\calP_p(\calX)$. This yields that Theorem 5.20 in \cite{villani_optimal_2008} is applicable. This theorem gives us the stability of the optimal solution as $\calP_p(\calX)$ is a closed subset of $\ell^1_{d^p}(\calX)$.
	\end{itemize}
So far, we checked all the assumptions of Theorem 4.24 in \cite{bonnans_perturbation_2000}. The rest of this section is devoted to the derivation of formula \eqref{eq:derivative} from the result of that theorem. \\
The Lagrangian $L$ of a parametrized optimization problem 
\[\min_{w} f(w,u) \text{ s.t. } C(w,u) \in K\]
is given by 
\[L(w,\lambda,u) = f(w,u) + \scalprod{\lambda,C(w,u)},\]
where $f$ is the objective function, $u$ the parameter and $C$ the constraint function and $\scalprod{\cdot,\cdot}$ the dual pairing (see for example Section 2.5.2 in \cite{bonnans_perturbation_2000}). We refer to $\lambda$ as Lagrange multiplier.  
For the transport problem this yields with $(\br,\bs)$ being the parameter and the definition of the constraint function in \eqref{eq:constraint_fun}
\begin{multline*}
L(\bw,(\bnu,\blambda, \bmu),(\br,\bs))  \\
=\sum_{x, x' \in \calX} d^p(x,x') w_{x,x'} + \scalprod{\bnu, \bw} + \scalprod{\blambda, \bw^T\one - \br} + \scalprod{\bmu, \bw\one - \bs}.\end{multline*}
Differentiating this in the Fr\'echet sense with respect to $(\br,\bs)$ and applying $(\bd_1,\bd_2)$ to this linear operator results in 
\[D_{(r,s)}L(\bw,(\bnu,\blambda,\bmu), (\br,\bs))(\bd_1,\bd_2) = -(\scalprod{\blambda, \bd_1} + \scalprod{\bmu, \bd_2})\] 
as the Lagrangian is linear and bounded in $(\br,\bs)$. As this derivative is independent of $\bw$ and the set of Lagrange multipliers $\Lambda(\br,\bs)$ equals the set of dual solutions $\calS^*(\br,\bs)$ in the case of a convex unperturbed problem (see section above Thm. 4.24 in \cite{bonnans_perturbation_2000}) it holds that the directional Hadamard derivative is given by
\begin{multline*}(\bd_1,\bd_2) \mapsto \inf_{\bw \in \calS(\br,\bs)} \sup_{(\blambda,\bmu) \in \Lambda(\br,\bs)} D_{(r,s)}L(\bw,(\bnu,\blambda,\bmu), (\br,\bs))(\bd_1,\bd_2) \\
= \inf_{\bw \in \calS(\br,\bs)} \sup_{(\blambda,\bmu) \in \Lambda(\br,\bs)} -(\scalprod{\blambda, \bd_1} + \scalprod{\bmu, \bd_2})  \\
=  \sup_{(\blambda,\bmu) \in \calS^*(\br,\bs)} -(\scalprod{\blambda, \bd_1} + \scalprod{\bmu, \bd_2}).\end{multline*}
\end{proof}

\subsection{The limit distribution under equality of measures}
\label{sec:limit_equality}

First, observe that for the case $\br = \bs$ the set of dual solutions $\calS^*(\br,\br)$ in \eqref{eq:dual_set_rs} reduces to:
\begin{align*}
 \calS^*(\br,\br)
&= \Big\{(\blambda,\bmu) \in \ell^{\infty}_{d^{-p}}(\calX)\times \ell^{\infty}_{d^{-p}}(\calX): \left\langle \br,\blambda\right\rangle +\left\langle \br,\bmu \right\rangle = 0,\\
&\hspace{35ex }\lambda_x + \mu_{x'} \leq d^p(x,x') \quad \forall x,x' \in \calX \Big\}\\
&=\Big\{(\blambda,\bmu) \in \ell^{\infty}_{d^{-p}}(\calX)\times \ell^{\infty}_{d^{-p}}(\calX):  \lambda_x = -\mu_x \text{ for } x \in \supp(\br), \\
&\hspace{35ex }\lambda_x + \mu_{x'} \leq d^p(x,x') \quad \forall x,x' \in \calX \Big\}.
\end{align*}

The equality follows as for $x = x'$ the inequality condition gives $\lambda_x + \mu_x \leq 0$ and all $r_x$ in the sum are non-negative. The conjunction of these two conditions yields $\lambda_x + \mu_x = 0$. \\
This set is a subset of the set given in \eqref{eq:dual_set}, but changing $\calS^*(r,r)$ to $\calS^*(r)$ does not change the optimal value of the linear programs in Theorem \ref{thm:distrlimit_one} and \ref{thm:distrlimit_two} as the Gaussian process $\bG$ is zero at all $x \notin \supp(\br)$.\\
In the case, that the support of $\br$, i.e., $\set{x \in \calX \colon r_x > 0}$, is the whole ground space $\calX$, the set $\calS^*(\br)$ is independent of $\br$ and it reduces to 
\begin{equation*}
\calS^* = \Big\{\blambda \in \ell^{\infty}_{d^{-p}}(\calX): \lambda_x - \lambda_{x'} \leq d^p(x,x') \quad \forall x,x' \in \calX \Big\}.
\end{equation*}
\begin{proof}[Proof of Thm. \ref{thm:distrlimit_two} a)] For the two sample case the delta method together with the continuous mapping theorem and equation \eqref{eq:two_sample_null} gives
\[\rho_{n,m}^{1/p} W_p(\brh_n, \bsh_m) \weak \left\lbrace \max_{(\blambda, \bmu) \in\calS^*(\br, \br)}  \sqrt{\alpha}\scalprod{\blambda,\bG} + \sqrt{1-\alpha} \scalprod{\bmu, \bG'}\right\rbrace^{1/p}.  \]
Nevertheless, for all $x \in \calX$ where $r_x > 0$ it holds $\lambda_x = - \mu_x$ and for all $x \in \calX$ where $r_x = 0$ the limit element $G_x$ is degenerate. Hence, the limit distribution above is equivalent in distribution to 
\[\left\lbrace \max_{\blambda \in\calS^*(\br, \br)}  \sqrt{\alpha}\scalprod{\blambda,\bG} - \sqrt{1-\alpha} \scalprod{\blambda, \bG'}\right\rbrace^{1/p}.  \]
The independence of $\bG$ and $\bG'$ yield that $\sqrt{\alpha} \scalprod{\blambda, \bG} - \sqrt{1-\alpha} \scalprod{\blambda, \bG'}$ equals $\sqrt{\alpha + (1- \alpha)} \scalprod{\blambda, \bG}$ in distribution and hence the limit reduces to
\[\left\lbrace \max_{\blambda \in\calS^*(\br)}  \scalprod{\blambda,\bG}\right\rbrace^{1/p}.\]
\end{proof}
\begin{proof}[Proof of decomposition in Rem. \ref{rem:theorem1} e)]
For the alternative representation of the distributional limit we decompose the Gaussian process $\bG$ with mean zero and covariance structure as defined in \eqref{eq:def_sigma} into $\bG = \bG^+ - \bG^-$ with $\bG^+$, $\bG^-$ non-negative, then the limiting distribution in \eqref{eq:one_null} can be rewritten as follows:
\[ \begin{aligned}
\max_{\blambda \in\calS^*(\br)} \scalprod{\bG,\blambda} &= \max_{\blambda \in\calS^*(\br,\br)} \scalprod{\bG^+,\blambda} - \scalprod{\bG^-,\blambda} \\
&= \max_{(\blambda, \bmu) \in \ell^\infty_{d^{-p}}(\calX) \times \ell^\infty_{d^{-p}}(\calX)} \scalprod{\bG^+,\blambda} + \scalprod{\bG^-,\bmu} \\
&~\text{s.t.} \quad \lambda_x + \mu_x = 0  \quad\text{for all}\quad x \in \supp(\br)\\
&\lambda_x + \mu_{x'} \leq d^p(x,x') \quad \forall x,x' \in \calX.
\end{aligned}
\]
The Lagrangian for this problem is given by 
\begin{multline*}
L(\blambda, \bmu, \bw, \bz) =  \sum_{x \in \calX}{G^+_x\lambda_x} + \sum_{x' \in \calX}{G^-_{x'}\mu_{x'}} \\
+ \sum_{x \in \calX}{z_x (\lambda_x + \mu_x) \one_{\{r_x > 0\}}} + \sum_{x,x' \in \calX}{w_{x,x'}(\lambda_x + \mu_{x'} - d^p(x,x'))}.
\end{multline*}
From this we can derive the dual via
\[\min_{\bw \geq 0 \in \ell^1_{d^p}(\calX \times \calX), \bz \in \ell^1_{d^p}(\calX)} \sup_{\blambda, \bmu \in \ell^\infty_{d^{-p}}(\calX)}  L(\blambda, \bmu, \bw, \bz),\]
where $\bw \geq 0$ to be understood componentwise. It yields
\begin{align*}
\inf_{\bw \geq 0, \bz} &\sum_{x,x' \in \calX} d^p(x,x') w_{x,x'} \\
\text{s.t.} \quad &\sum_{x' \in \calX} w_{x,x'} = G^+_x + z_x\one_{\{r_x > 0\}} \\
&\sum_{x \in \calX} w_{x,x'} = G^-_{x'} +z_{x'}\one_{\{r_x > 0\}},
\end{align*}
where the minimum over $\bw$ equals the $p$-th power of the $p$-th Wasserstein distance. More precisely the linear program above is equivalent to 
\[\inf_{\bz(\br) \in \ell^1_{d^p}(\calX)} W_p^p\left(\bG^+ + \bz(\br), \bG^- + \bz(\br) \right), \]
where $\bz(\br)$ depends on $\br$ through the support of $\br$ in the following sense: $z_x = 0$ for $x \in \calX$ such that $r_x = 0$.
\end{proof}

\subsection{Proof of Theorem \ref{THM:TREES}}
\label{sec:proof_tree}
\paragraph{Simplify the set of dual solutions $\calS^*$} As a first step, we
rewrite the set of dual solutions $\calS^*$ given in definition 
\eqref{eq:dual_full} in our tree notation as
\begin{equation}
  \calS^* = \left\{ \blambda\in\ell^\infty_{d^{-p}}(\calX): \lambda_x - \lambda_{x'} \leq
d_\calT(x,x')^p, \quad x,x'\in \calX \right\}.
\label{eq:S*_tree}
  \end{equation}
  The key observation is that in the condition $\lambda_x - \lambda_{x'}\leq d_\calT(x,x')^p$ we
  do not need to consider all pairs of vertices $x,x'\in \calX$, but only those which
  are joined by an edge. To see this, assume that only the latter condition holds.
  Let $x,x'\in \calX$ arbitrary and $x = x_1,
  \dots , x_n = x'$ the sequence of vertices defining the unique path joining
  $x$ and $x'$, such that $(x_j,x_{j+1})\in E$ for $j=1,\dots,n-1$. That this path contains only a finite number of edges, was proven in Section \ref{sec:tree}. Then
  \[
    \lambda_x - \lambda_{x'} = \sum_{j=1}^{n-1} (\lambda_{x_j} - \lambda_{x_{j+1}}) \leq \sum_{j=1}^{n-1}
    d_\calT(x_j, x_{j+1})^p \leq d_\calT(x,x')^p,
  \]
  such that \eqref{eq:S*_tree} is satisfied for all $x,x'\in \calX$. Noting that if two
  vertices are joined by an edge then one has to be the parent of the other, we
  can write the set of dual solutions as
  \begin{equation}
    \calS^* = \left\{ \blambda\in \ell^\infty_{d^{-p}}(\calX)  : |\lambda_x -
      \lambda_{\parent(x)}| \leq d_\calT(x,\parent(x))^p ,\quad x\in \calX  \right\}.
      \label{eq:S*_trees_simple}
    \end{equation}

    \paragraph{Rewrite the target function} 
    To rewrite the target function we need to make several definitions.
    Let \[
    \tilde{e}^{(x)}_y = \begin{cases}
    \frac{1}{d^p(x,x_0)} & \text{ if } y = x,\\
    -\frac{1}{d^p(x,x_0)} & \text{ if } y = \parent(x),\\
    0 & \text{ else.}
    \end{cases}\]
    Furthermore, we define for $\bmu \in \ell^1_{d^p}(\calX)$
    \[\eta_x = \sum_{x' \in \children(x)}d^p(x,x_0) \mu_{x'}\]
    and 
    \[ \begin{aligned}
    \bmu_n &= \sum_{x \in A_{\leq n}\setminus \rootT(\calT) } \eta_x \tilde{ \be}^{(x)} = \bmu \one_{A_{<n} } + \sum_{x \in A_{=n}} \frac{1}{d^p(x,x_0)} \eta_x \be{(x)},
    \end{aligned}\] 
    here $A_{\leq n} = \set{x \in \calX\colon \text{level of }x \leq n, x \text{ is within the first } n \text{ vertices of its level}}$, \\
    $A_{= n} = \set{x \in \calX\colon \text{level of }x = n, x \text{ is within the first } n \text{ vertices of its level}}$, \\
    $A_{> n} = \set{x \in \calX\colon \text{level of }x > n  \text{ or } x \text{ is not within the first } n \text{ vertices of its level}}$ and 
    $\be{(x)}$ the sequence 1 at $x$ and 0 everywhere else.
    For this sequence $\bmu_n$ it holds
    \[\begin{aligned}
    \normfp{\bmu - \bmu_n} &= \sum_{x \in X} d^p(x,x_0) \abs{\bmu \one_{A_{>n}} - \sum_{\tilde{x} \in A_{=n}} \frac{1}{d^p(\tilde{x}, x_0)}\eta_{\tilde{x}} \be^{(\tilde{x})}}_x \\
    & \leq \normfp{\bmu \one_{A_{>n}}} + \abs{\sum_{x \in A_{=n}} \eta_x}.
    \end{aligned}\]
   As $n \to \infty$, the first part tends to zero as $\bmu \in \ell^1_{d^p}(\calX)$, and 
   \[\abs{\sum_{x \in A_{=n}} \eta_x} \leq \sum_{x \in A_{=n}} \sum_{x' \in \children(x)} \abs{\mu_{x'}} d^p(x',x_0) \leq \sum_{x \in A_{\geq n}} \abs{\mu_x} d^p(x,x_0) \xrightarrow{n \to \infty} 0. \]
   Hence, our target function for $\bmu \in \ell^1_{d^p(\calX)}$ and $\blambda \in \ell^{\infty}_{d^{-p}}(\calX)$ can be rewritten in the following way 
   \begin{equation}\begin{aligned}
   \scalprod{\bmu, \blambda} &= \lim_{n \to \infty} \scalprod{\bmu_n, \blambda} \\
   & = \lim_{n \to \infty} \sum_{x \in A_{\leq n}} \eta_x \scalprod{\tilde{\be}^{(x)}, \blambda} \\
   & = \lim_{n \to \infty} \sum_{x \in A_{\leq n}}  \sum_{x' \in \children(x)} \mu_{x'}(\lambda_x - \lambda_{\parent(x)}) \\
   & \leq \lim_{n \to \infty} \sum_{x \in A_{\leq n}}  \abs{\sum_{x' \in \children(x)} \mu_{x'}}\abs{\lambda_x - \lambda_{\parent(x)}}\\
   & = \lim_{n \to \infty} \sum_{x \in A_{\leq n}}  \abs{(S_\calT \bmu)_x}\abs{\lambda_x - \lambda_{\parent(x)}}
    \end{aligned}
    \label{eq:retarget}
    \end{equation}
    Observe that for $\blambda \in \calS^*$ it holds 
    \begin{equation}\abs{\lambda_x - \lambda_{\parent(x)}} \leq d^P(x,\parent(x)). 
    \label{eq:ineq_dual}
    \end{equation}
    By condition \eqref{eq:entropy} $\bm{G}\sim\mathcal{N}(0,\Sigma(\br))$ is an element of $\ell^1_{d^p}(\calX)$. For $\blambda\in\calS^*$ we get with \eqref{eq:retarget} and \eqref{eq:ineq_dual} that 
    \begin{equation}
      \scalprod{\bm{G},\blambda} \leq \lim_{n \to \infty} \sum_{x \in A_{\leq n}}  \abs{(S_\calT \bG)_x}d_\calT(x,\parent(x))^p.
      \label{eq:tree_bound_on_max}
    \end{equation}
    Therefore, $\max_{\blambda\in\calS^*} \scalprod{\bm{G},\blambda}$ is bounded by
    $\lim_{n \to \infty} \sum_{x \in A_{\leq n}}  |(S_\calT \bm{G})_x|
    d_\calT(x,\parent(x))^p$. We can define the sequence
    $\bnu\in\ell^\infty_{d^{-p}}(\calX)$ by 
    \begin{equation}
   \begin{aligned}
   	\nu_{\rootT}& = 0 \\
   	\nu_{x} - \nu_{\parent(x)}& = \mathrm{sign}((S_{\calT}\bG)_x) d_{\calT}(x,\parent(x))^p \\ 
   \end{aligned}
   \end{equation}
    From \eqref{eq:S*_trees_simple} and the fact that $d^p(x,\parent(x)) \leq d^p(x,\rootT(\calT))$ we see that $\bnu\in\calS^*$ and by plugging $\bnu$ into equation \eqref{eq:tree_bound_on_max} we can conclude that $\scalprod{\bm{G},\bnu}$ attains the upper bound in
    \eqref{eq:tree_bound_on_max}. \\
    As the last step of our proof, we verify that the limit in \eqref{eq:tree_bound_on_max} exists.
    Therefore, we rewrite condition \eqref{eq:entropy} in terms of the edges and recall that $x_0 = \rootT(\calT)$.
	\begin{equation}
	\sum_{x \in \calX} d_\calT(x,x_0)^p \sqrt{r_x} \geq \sum_{x\in \calX} \sum_{x' \in \children(x)} d_\calT(x,\parent(x))^p \sqrt{r_{x'}}.
	\label{eq:tree_entropy}
	\end{equation}
	The first moment of the limiting distribution can be bounded in the following way:
	\begin{equation*}
	\begin{aligned}
		&\mean \left[\sum_{ x \in \calX\setminus\{\rootT(\calT)\}} |(S_\calT \bm{G})_x| d_\calT(x,\parent(x))^p \right] \\
		&\qquad \leq \sum_{x \in \calX} d_\calT(x,\parent(x))^p \sqrt{(S_\calT r)_x(1-(S_\calT r)_x)} \\
		& \qquad \leq \sum_{x\in \calX} \sum_{x' \in \children(x)} d_\calT(x,\parent(x))^p \sqrt{r_{x'}} \\
		& \qquad< \infty
	\end{aligned}
	\end{equation*}
	due to Hölder's inequality and \eqref{eq:tree_entropy}. This bound shows that the limit in \eqref{eq:tree_bound_on_max} is almost surely finite and hence, concludes the proof.